%% file: main.tex
\pdfoutput=1
\documentclass[a4paper, oneside, 12pt, reqno]{amsart}
\input{packages-macros.tex}

\begin{document}

\title[]{The Grothendieck construction for delta lenses}
\author{Bryce Clarke}
\address{Tallinn University of Technology, Tallinn, Estonia}
\email{\href{mailto:bryce.clarke@taltech.ee}{bryce.clarke@taltech.ee}}
\urladdr{\href{https://bryceclarke.github.io}{bryceclarke.github.io}}
\date{}

\begin{abstract}
Delta lenses are functors equipped with a functorial choice of lifts, generalising the notion of split opfibration. 
In this paper, we introduce a Grothendieck construction (or category of elements) for delta lenses, thus demonstrating a correspondence between delta lenses and certain lax double functors into the double category of sets, functions, and split multivalued functions. 
We show that the double category of split multivalued functions admits a universal property as a certain kind of limit, and inherits many nice properties from the double category of spans. 
Applications of this construction to the theory of delta lenses are explored in detail.
\end{abstract}

\maketitle
\thispagestyle{empty} 

\input{introduction.tex}
\input{delta-lenses.tex}
\input{double-categories.tex}
\input{category-of-elements.tex}
\input{examples.tex}

\input{references.tex}

\end{document}

%% file: packages-macros.tex
\usepackage[utf8]{inputenc}
\usepackage[T1]{fontenc}
\usepackage{mlmodern}
\usepackage{microtype}

\usepackage[text={148mm,210mm}, centering, footskip = 1cm]{geometry}

\usepackage{enumitem}

\usepackage{mathtools, amssymb, amsthm}
\usepackage[mathcal]{eucal}

\usepackage{tikz}
\usepackage{tikz-cd}
\usepackage{quiver}

\usepackage{bbm}

\usepackage[colorlinks = true]{hyperref}
\hypersetup{allcolors=[rgb]{0.5,0,0}}

\newtheorem{theorem}{Theorem}
\newtheorem{proposition}[theorem]{Proposition}
\newtheorem{lemma}[theorem]{Lemma}
\newtheorem{corollary}[theorem]{Corollary}
\newtheorem*{unmarked-theorem}{Theorem}

\theoremstyle{definition} 
\newtheorem{definition}[theorem]{Definition}
\newtheorem{example}[theorem]{Example}
\newtheorem{remark}[theorem]{Remark}

\renewcommand{\AA}{\mathbb{A}}
\newcommand{\BB}{\mathbb{B}}
\newcommand{\CC}{\mathbb{C}}
\newcommand{\DD}{\mathbb{D}}
\newcommand{\XX}{\mathbb{X}}
\newcommand{\YY}{\mathbb{Y}}

\newcommand{\SPAN}{\mathrm{\mathbb{S}pan}}
\newcommand{\SMULT}{\mathrm{\mathbb{S}Mult}}
\newcommand{\LO}{\mathrm{\mathbb{L}o}}
\newcommand{\SQ}{\mathrm{\mathbb{S}q}}
\newcommand{\MOD}{\mathrm{\mathbb{M}od}}
\newcommand{\PROF}{\mathrm{\mathbb{P}rof}}

\newcommand{\el}{\mathcal{E}\ell}
\newcommand{\A}{\mathcal{A}}
\newcommand{\B}{\mathcal{B}}
\newcommand{\C}{\mathcal{C}}
\newcommand{\D}{\mathcal{D}}
\newcommand{\X}{\mathcal{X}}
\newcommand{\Y}{\mathcal{Y}}
\newcommand{\CAT}{\mathrm{\mathcal{C}AT}}

\newcommand{\Set}{\mathrm{\mathcal{S}et}}
\newcommand{\Cat}{\mathrm{\mathcal{C}at}}
\newcommand{\Lens}{\mathrm{\mathcal{L}ens}}
\newcommand{\RetFun}{\mathrm{\mathcal{R}etFun}}
\newcommand{\DiaLens}{\mathrm{\mathcal{D}iaLens}}
\newcommand{\Dbl}{\mathrm{\mathcal{D}bl}}
\newcommand{\DOpf}{\mathrm{\mathcal{D}Opf}}
\newcommand{\SOpf}{\mathrm{\mathcal{S}Opf}}
\newcommand{\Idx}{\mathrm{\mathcal{I}dx}}

\newcommand{\GlobCone}{\mathrm{\mathcal{G}lobCone}}

\DeclareMathOperator{\El}{El}

\newcommand{\two}{\mathbf{2}}

\DeclareMathOperator{\dom}{dom}
\DeclareMathOperator{\cod}{cod}
\DeclareMathOperator{\id}{id}
\DeclareMathOperator{\op}{op}
\DeclareMathOperator{\obj}{obj}

\DeclareMathOperator{\Dec}{Dec}
\DeclareMathOperator{\lax}{lax}

\makeatletter
\def\slashedarrowfill@#1#2#3#4#5{%
$\m@th\thickmuskip0mu\medmuskip\thickmuskip\thinmuskip\thickmuskip
\relax#5#1\mkern-7mu%
\cleaders\hbox{$#5\mkern-2mu#2\mkern-2mu$}\hfill
\mathclap{#3}\mathclap{#2}%
\cleaders\hbox{$#5\mkern-2mu#2\mkern-2mu$}\hfill
\mkern-7mu#4$%
}
\def\rightslashedarrowfill@{%
\slashedarrowfill@\relbar\relbar\mapstochar\rightarrow}
\newcommand\xslashedrightarrow[2][]{%
\ext@arrow 0055{\rightslashedarrowfill@}{#1}{#2}}
\def\leftslashedarrowfill@{%
\slashedarrowfill@\leftarrow\relbar\mapsfromchar\relbar}
\newcommand\xslashedleftarrow[2][]{%
\ext@arrow 0055{\leftslashedarrowfill@}{#1}{#2}}
\makeatother

\newcommand{\xlto}{\xslashedrightarrow}
\newcommand{\lto}{\xlto{}}

%% file: introduction.tex
\section{Introduction}

Delta lenses were introduced in 2011 by Diskin, Xiong, and Czarnecki \cite{DiskinXiongCzarnecki2011} as an algebraic framework for bidirectional transformations \cite{AbouSalehCheneyGibbonsMcKinnaStevens2018, CzarneckiFosterHuLammelSchurrTerwilliger2009}.
In 2013, Johnson and Rosebrugh \cite{JohnsonRosebrugh2013} showed that every split opfibration has an underlying delta lens, thus initiating an ongoing study of delta lenses using category theory. 
The Grothendieck construction \cite{GrothendieckRaynaud1971} is widely recognised as one of the most important concepts in category theory, and establishes a correspondence between split opfibrations and functors into the category $\Cat$ of (small) categories and functors. 
Given the close connection between delta lenses and split opfibrations, it is natural to ask: is there a Grothendieck construction for delta lenses?  

The central contribution of this work is a correspondence, as displayed below, between delta lenses and certain lax double functors into the double category of sets, functions, and split multivalued functions.
\begin{center}
    delta lenses $\A \to \B$ 
    \qquad $\leftrightsquigarrow$ \qquad
    lax double functors $\LO(\B) \to \SMULT$
\end{center}
This correspondence extends to an equivalence of categories (Theorem~\ref{theorem:main}).

\subsection{Overview of the main result}
\label{subsec:overview}

We now provide a concise overview of the main result and introduce the necessary technical details to understand it. 

A \emph{delta lens} is a pair $(f, \varphi)$ consisting of a functor $f \colon \A \to \B$ equipped with a lifting operation that provides, for each object $a \in \A$ and morphism $u \colon fa \to b$ in~$\B$, a morphism $\varphi(a, u) \colon a \to a'$ in $\A$, called a \emph{chosen lift}, such that $f \varphi(a, u) = u$. 
The lifting operation is required to preserve identities and composition of morphisms. 
Let $\Lens$ denote the category whose objects are delta lenses, and whose morphisms $(h, k) \colon (f, \varphi) \to (g, \psi)$ are pairs of functors that commute with the underlying functors and preserve the chosen lifts; that is, $kf = gh$ and $h\varphi(a, u) = \psi(ha, ku)$. 

A \emph{split multivalued function} $A \lto B$ is a diagram of functions as below. 
\begin{equation}
\label{equation:smf-intro}
    \begin{tikzcd}
        A
        & 
        X 
        \arrow[l, two heads, "s", shift left = 2]
        \arrow[from=l, tail, "\sigma", shift left = 2]
        \arrow[r, "t"]
        & 
        B
    \end{tikzcd}
    \qquad \qquad
    s \circ \sigma = 1_{A}
\end{equation}
A lax double functor $F \colon \LO(\B) \to \SMULT$, where $\B$ is a category, is called an \emph{indexed split multivalued function}\footnote{To be more precise, it should be called a \emph{covariant lax $\Cat$-indexed split multivalued function}.}. 
We can unpack this data without mentioning double category theory. 
An indexed split multivalued function $(\B, F)$ consists of a category $\B$, a set $F(x)$ for each object $x \in \B$, and a split multivalued function 
\[
    \begin{tikzcd}
        F(x)
        & 
        F(u) 
        \arrow[l, two heads, "s_{u}", shift left = 2]
        \arrow[from=l, tail, "\sigma_{u}", shift left = 2]
        \arrow[r, "t_{u}"]
        & 
        F(y)
    \end{tikzcd}
    \qquad \qquad
    s_{u} \circ \sigma_{u} = 1_{F(x)}
\]
for each morphism $u \colon x \to y$ in $\B$. 
There is also a function $\mu_{(u, v)} \colon F(u, v) \to F(v \circ u)$ for each composable pair $(u \colon x \to y, v \colon y \to z)$ in $\B$, where $F(u, v)$ is the pullback of $t_{u}$ and $s_{v}$. 
This data is required to satisfy several natural axioms. 

A morphism $(k, \theta) \colon (\B, F) \to (\D, G)$ of indexed split multivalued functions consists of a functor $k \colon \B \to \D$, a function $\theta_{x} \colon F(x) \to G(kx)$ for each object $x \in \B$, and a function $\theta_{u} \colon F(u) \to G(ku)$ for each morphism $u \colon x \to y$ in $\B$. 
This data also satisfies several natural axioms. 
Let $\Idx(\Cat, \SMULT)$ denote the category of indexed split multivalued functions and morphisms between them. 

\begin{unmarked-theorem}
    There is an equivalence of categories 
    \[
        \el \colon \Idx(\Cat, \SMULT) \xrightarrow{\hphantom{-}\simeq{\hphantom{-}}} \Lens
    \] 
    between the categories of indexed split multivalued functions and delta lenses. 
\end{unmarked-theorem}

The functor $\el \colon \Idx(\Cat, \SMULT) \to \Lens$ describes the \emph{Grothendieck construction} for delta lenses, or alternatively, the \emph{category of elements} of an indexed split multivalued function. 

\subsection{A proof sketch}
\label{sec:proof-sketch}
It is not too difficult to see why this result should be true. 

Given an indexed split multivalued function $(\B, F)$, we may construct a category $\el(\B, F)$ whose objects are pairs $(x \in \B, a \in F(x))$ and morphisms are pairs $(u \colon x \to y \in \B, \alpha \in F(u)) \colon (x, s_{u}(\alpha)) \to (y, t_{u}(\alpha))$.
The identity on an object $(x, a)$ is the morphism $(1_{x}, \sigma_{1_{x}}(a))$, and composition of morphisms is determined by the functions $\mu_{(u, v)}$. 
There is a functor $\pi \colon \el(\B, F) \to \B$ given by projection in the first component, and this admits a delta lens structure, since for each object $(x, a) \in \el(\B, F)$ and morphism $u \colon x \to y$ in $\B$, there is a chosen lift $(u, \sigma_{u}(a)) \colon (x, a) \to (y, t_{u}\sigma_{u}(a))$. 

Conversely, given a delta lens $(f, \varphi) \colon \A \to \B$, we may construct an indexed split multivalued function $(\B, F)$ consisting of the fibre sets $F(x) = \{ a \in \A \mid fa = x \}$ and $F(u) = \{ \alpha \colon a \to a' \mid f\alpha = u \}$ for each object $x \in \B$ and morphism $u \colon x \to y$ in~$\B$. 
The functions $s_{u}$, $t_{u}$ and $\mu_{(u, v)}$ are given by the source, target, and composition operations of $\A$, while $\sigma_{u} \colon F(x) \rightarrowtail F(u)$ is determined by the lifting operation $a \mapsto \varphi(a, u)$ of the delta lens.

The key insight here is that the lifting operation $\varphi$ of a delta lens is the same as a choice of section $\sigma_{u}$ of the ``source leg'' of an indexed span $F(x) \leftarrow F(u) \rightarrow F(y)$, such that these choices  respect the identities and composition in $\B$. 
If we put aside this data, we find the well-known correspondence between functors and certain lax functors into the double category of sets, functions, and spans \cite{AhrensLumsdaine2019,Pare2011,PavlovicAbramsky1997}.
\[
    \Idx(\Cat, \SPAN) \simeq \Cat^{\two}
\]

Rather than proceeding with our proof sketch to establish the equivalence between the categories of indexed split multivalued functions and delta lenses, which would involve checking many tedious details, we instead leverage the equivalence above construct an abstract proof. 
The focus is to develop a clear understanding of the Grothendieck construction for delta lenses using double category theory. 

\subsection{Outline and contributions}
We now provide an outline of the paper and details of the significant contributions and results. 

In Section~\ref{sec:delta-lenses}, we construct an equivalence $\Lens \simeq \DiaLens$ between the categories of delta lenses and diagrammatic delta lenses. 
A \emph{diagrammatic delta lens} $(f, p)$ is simply a commuting diagram of functors 
\begin{equation*}
    \begin{tikzcd}[column sep = small]
        \X && \A \\
        & \B
        \arrow["p", from=1-1, to=1-3]
        \arrow["fp"', from=1-1, to=2-2]
        \arrow["f", from=1-3, to=2-2]
    \end{tikzcd}
\end{equation*}
such that $p$ is \emph{identity-on-objects} and $fp$ is a \emph{discrete opfibration}. 
The correspondence between delta lenses and diagrammatic delta lenses was introduced by the author in 2019 (see \cite[Corollary~20]{Clarke2020} and \cite[Proposition~3.7/3.8]{Clarke2020b}), however an explicit statement of the equivalence did not appear until much later \cite[Proposition~2.15]{Clarke2024}. 
Unfortunately, despite the importance and usefulness of this result, 
a detailed proof in the literature has remained absent; we now fill this gap in Theorem~\ref{theorem:Lens-DiaLens}. 

In Section~\ref{sec:SMULT}, we introduce the double category $\SMULT$ of sets, functions, and split multivalued functions. 
Closely related to $\SMULT$ is the double category $\SQ(\Set)$ whose cells are commutative squares of functions, and the double category $\SPAN$ of sets, functions, and spans. 
We show that there is a diagram of double functors
\begin{equation*}
    \begin{tikzcd}[column sep = small]
        & \SMULT \\
        {\SQ(\Set)} && \SPAN
        \arrow[""{name=0, anchor=center, inner sep=0}, "{U_{1}}"', from=1-2, to=2-1]
        \arrow[""{name=1, anchor=center, inner sep=0}, "{U_{2}}", from=1-2, to=2-3]
        \arrow["K_{\ast}"', from=2-1, to=2-3]
        \arrow["\sigma", shift right=2, shorten <=10pt, shorten >=10pt, Rightarrow, from=0, to=1]
    \end{tikzcd}
\end{equation*}
where $K_{\ast}$ sends a function $f \colon A \to B$ to its companion span $(1_{A}, A, f) \colon A \lto B$, and a split multivalued function \eqref{equation:smf-intro} is sent to its underlying function $t \circ \sigma \colon A \to B$ by $U_{1}$ and its underlying span $(s, X, t) \colon A \lto B$ by $U_{2}$.
A \emph{globular transformation} between double functors is defined as analogously to that of an \emph{icon} between morphisms of bicategories \cite{Lack2010}, and $\sigma \colon K_{\ast} U_{1} \Rightarrow U_{2}$ is a globular transformation whose component at a split multivalued function \eqref{equation:smf-intro} is given by the following cell in $\SPAN$. 
\begin{equation*}
    \begin{tikzcd}
        A & A & B \\
        A & X & B
        \arrow["{1_{A}}"', from=1-1, to=2-1]
        \arrow["{1_{A}}"', from=1-2, to=1-1]
        \arrow["{t \sigma}", from=1-2, to=1-3]
        \arrow["\sigma", from=1-2, to=2-2]
        \arrow["{1_{B}}", from=1-3, to=2-3]
        \arrow["s", from=2-2, to=2-1]
        \arrow["t"', from=2-2, to=2-3]
    \end{tikzcd}
\end{equation*}

The main contribution in Section~\ref{sec:SMULT} is characterising $\SMULT$ has a certain kind of limit --- the terminal globular cone over $K_{\ast}$ --- and demonstrate that it has both a $1$-dimensional and $2$-dimensional universal property (Theorem~\ref{theorem:1d-universal-property-SMULT} and Theorem~\ref{theorem:2d-universal-property-SMULT}). 
As a corollary, we obtain an isomorphism $\GlobCone(\Dbl, K_{\ast}) \cong \Idx(\Dbl, \SMULT)$ between categories whose objects are globular cones over $K_{\ast} \colon \SQ(\Set) \to \SPAN$ and lax double functors into $\SMULT$, respectively.

There is a fully faithful functor $\LO \colon \Cat \to \Dbl$ which constructs, for each category $\B$, a strict double category $\LO(\B)$ whose objects and loose morphisms are the objects and morphisms of $\B$, and whose tight morphisms and cells are identities. 
Using this embedding we can construct an isomorphism $\GlobCone(\Cat, K_{\ast}) \cong \Idx(\Cat, \SMULT)$, 
yielding a bijection between globular cones and lax double functors as follows.
\begin{equation*}
    \begin{tikzcd}[column sep = small]
        & \LO(\B) \\
        {\SQ(\Set)} && \SPAN
        \arrow[""{name=0, anchor=center, inner sep=0}, "{F_{1}}"', from=1-2, to=2-1]
        \arrow[""{name=1, anchor=center, inner sep=0}, "{F_{2}}", from=1-2, to=2-3]
        \arrow["K_{\ast}"', from=2-1, to=2-3]
        \arrow["\varphi", shift right=2, shorten <=10pt, shorten >=10pt, Rightarrow, from=0, to=1]
    \end{tikzcd}
    \qquad 
    \leftrightsquigarrow
    \qquad
    \begin{tikzcd}
        {\LO(\B)} \\
        \SMULT
        \arrow["F"', from=1-1, to=2-1]
    \end{tikzcd}
\end{equation*}

In Section~\ref{sec:main-theorem}, we demonstrate a correspondence between globular cones as above and diagrammatic delta lenses, then assemble a series of equivalences to prove our main theorem.
As recalled in Subsection~\ref{sec:proof-sketch}, there is a correspondence between functors into $\B$ and lax double functors $\LO(\B) \to \SPAN$. 
We can see that lax double functors $\LO(\B) \to \SPAN$ which factor through $K_{\ast} \colon \SQ(\Set) \to \SPAN$ are the same as functors $\B \to \Set$. 
This observation together with the classical ``category of elements'' construction yields an equivalence $\Idx(\Cat, \SQ(\Set)) \simeq \DOpf$ between categories whose objects are (strict) double functors $\LO(\B) \to \SQ(\Set)$ and \emph{discrete opfibrations}, respectively.
In Lemma~\ref{lemma:globular-to-ioo} we show a correspondence between \emph{identity-on-objects} morphisms in $\Cat / \B$ and globular transformations between lax double functors $\LO(\B) \to \SPAN$. 
Using these facts, we obtain the desired equivalence $\DiaLens \simeq \GlobCone(\Cat, K_{\ast})$ in Theorem~\ref{theorem:DiaLens-GlobCone}.

The central contribution of Section~\ref{sec:main-theorem} and the main result of the paper, Theorem~\ref{theorem:main}, follows immediately from the established sequence of equivalences. 
\[
    \Lens \simeq \DiaLens \simeq \GlobCone(\Cat, K_{\ast}) \cong \Idx(\Cat, \SMULT)
\]

In Section~\ref{sec:examples}, we examine several consequences, examples, and applications of our main result to the theory of delta lenses, including: 
\begin{itemize}
    \item Using adjunctions with $\SMULT$ to recover adjunctions with $\Lens$; 
    \item Characterising split opfibrations as indexed split multivalued functions; 
    \item Showing that delta lenses correspond to normal lax double functors into the double category $\MOD(\SMULT)$ of bimodules;
    \item Examining the pullback and pushforward of delta lenses;
    \item Characterising \emph{retrofunctors} as certain indexed split multivalued functions;
    \item Proving that the functor sending an indexed split multivalued function $(\B, F)$ to its category of elements $\el(\B, F)$ is a left adjoint;
    \item Exhibiting monoidal structures on $\Lens$, and the fibres of the codomain functor $\cod \colon \Lens \to \Cat$, as arising from monoidal structures on $\SMULT$;
    \item Characterising natural classes of delta lenses via factorisations through sub-double categories of $\SMULT$.
\end{itemize}

\subsection{Context and related work}
This work sits at the intersection of research on delta lenses, Grothendieck constructions, and double categories. 

The notion of delta lens was preceded by the concept of \emph{state-based lens}, which consists of a pair of functions $A \to B$ and $A \times B \to A$ satisfying three axioms~\cite{FosterGreenwaldMoorePierceSchmitt2007}.
State-based lenses were later shown to be the same as delta lenses between \emph{codiscrete categories} \cite{JohnsonRosebrugh2016}. 
In their seminal paper \cite{JohnsonRosebrugh2013}, Johnson and Rosebrugh showed that delta lenses capture the underlying structure of split opfibrations, and characterised them as certain algebras for a \emph{semi-monad} (an endofunctor $T$ with an associative multiplication $\mu \colon T^{2} \Rightarrow T$) on the slice category $\Cat / \B$, connecting with the characterisations of state-based lenses \cite{JohnsonRosebrughWood2010} and split opfibrations \cite{Street1974} as algebras.

Delta lenses have typically been treated as \emph{morphisms} between categories due to their applications in computer science. 
The properties of the category of (small) categories and delta lenses have been studied in several places \cite{CholletClarkeJohnsonSongaWangZardini2022,DiMeglio2022,DiMeglio2023}, although these are usually less nice than in $\Cat$. 
Ahman and Uustalu \cite{AhmanUustalu2017} characterised delta lenses as a compatible functor and \emph{retrofunctor} pair. 
Retrofunctors are another natural notion of morphism between categories, introduced under the name \emph{cofunctor} by Aguiar \cite{Aguiar1997}, and arise both as morphisms of monads in $\SPAN$ \cite{Pare2024} and morphisms of polynomial comonads on $\Set$ \cite{AhmanUustalu2016, SpivakGarnerFairbanks2025}.
The link between delta lenses and retrofunctors led to their representation as certain commutative diagrams in $\Cat$, and their relationship with split opfibrations in internal category theory \cite{Clarke2020,Clarke2020b}. 

The treatment of delta lenses as \emph{objects} in a category $\Lens$ has also proved to be important. 
This approach is necessary for the characterisation of delta lenses as retrofunctors with additional \emph{coalgebraic structure} \cite{Clarke2021}, again linking with similar results characterising state-based lenses \cite{GibbonsJohnson2012} and split opfibrations \cite{EmmeneggerMesitiRosoliniStreicher2024} as coalgebras for a comonad. 
Delta lenses were also characterised as algebras for a monad \cite{Clarke2023} on $\Cat^{\two}$, refining the previous result of Johnson and Rosebrugh, and reaffirming the view that they are functors with additional \emph{algebraic structure}. 

Ultimately, delta lenses should be considered as both objects and morphisms, an approach supported by studying the double category of categories, functors, and delta lenses \cite{Clarke2022}.
Like split opfibrations \cite{GrandisTholen2006}, delta lenses are not just algebras for a monad, but the right class of an \emph{algebraic weak factorisation system} \cite{Clarke2024}, naturally inducing double category in which to study them \cite{BourkeGarner2017}. 

There have been numerous variants of the Grothendieck construction introduced for different purposes \cite{Buckley2014, CruttwellLambertPronkSzyld2022, Lambert2021, Manuell2022, MoellerVasilakopoulou2020, Myers2021}, and this paper contributes to the recent interest in double-categorical generalisations of this concept. 

%% file: delta-lenses.tex
\section{Delta lenses and diagrammatic delta lenses}
\label{sec:delta-lenses}

In this section, we recall the definitions of delta lens, discrete opfibration, and diagrammatic delta lens. 
Our main result is an equivalence $\Lens \simeq \DiaLens$ between the categories of delta lenses and diagrammatic delta lenses (Theorem~\ref{theorem:Lens-DiaLens}). 
We recall the characterisation of split opfibrations as delta lenses, or diagrammatic delta lenses, with a certain property. 

\subsection{Delta lenses and discrete opfibrations}

Recall that a \emph{discrete opfibration} is a functor $f \colon \A \to \B$ such that, for each object $a \in \A$ and morphism $u \colon fa \to b$ in $\B$, there exists a unique morphism $w \colon a \to a'$ in $\A$ such that $fw = u$. 
Delta lenses are a generalisation of discrete opfibrations in which a  functor is equipped with a functorial choice of lifts without the requirement that these lifts are unique.

\begin{definition}
\label{definition:delta-lens}
    A \emph{delta lens} $(f, \varphi) \colon \A \to \B$ is a functor $f \colon \A \to \B$ equipped with a \emph{lifting operation}
    \[
        (a \in \A, u \colon fa \to b \in \B) 
        \quad \longmapsto \quad
        \varphi(a, u) \colon a \to a' \in \A 
    \]
    such that the following axioms hold: 
    \begin{enumerate}[label=(DL\arabic*)]
    \itemsep=1ex
        \item \quad $f\varphi(a, u) = u$; \label{DL1}
        \item \quad $\varphi(a, 1_{fa}) = 1_{a}$; \label{DL2}
        \item \quad $\varphi(a, v \circ u) = \varphi(a', v) \circ \varphi(a, u)$. \label{DL3}
    \end{enumerate}
\end{definition}

The first axiom \ref{DL1} of a delta lens $(f, \varphi)$ allows us to refer to the morphisms $\varphi(a, u)$ as \emph{chosen lifts} with respect to the functor $f$. 
The second axiom \ref{DL2} states that the lifting operation preserves identities, and the third axiom \ref{DL3} states that it preserves composition. 
Using the axioms of a delta lens, one can prove the following simple result. 

\begin{lemma}
\label{lemma:Lambda-category}
    Given a delta lens $(f, \varphi) \colon \A \to \B$, there is a wide subcategory $\Lambda(f, \varphi)$ of $\A$ whose morphisms are the chosen lifts. 
    By definition, the subcategory inclusion $\overline{\varphi} \colon \Lambda(f, \varphi) \to \A$ is faithful and identity-on-objects.
    Furthermore, the composite functor $f \, \overline{\varphi} \colon \Lambda(f, \varphi) \to \B$ is a discrete opfibration. 
\end{lemma}

This lemma clarifies the exact relationship between delta lenses and discrete opfibrations. 
Crucially, Lemma~\ref{lemma:Lambda-category} admits a converse which allows us to construct delta lenses from certain commutative triangles of functors.
We first examine the special case of discrete opfibrations. 

\begin{example}
\label{example:discrete-opfibration}
    Each discrete opfibration $f \colon \A \to \B$ admits a unique delta lens structure which we denote $(f, \Upsilon_{f}) \colon \A \to \B$. 
    The lifting operation $\Upsilon_{f}$ is well-defined by the uniqueness of lifts.  
    Furthermore, since every morphism in the domain of a discrete opfibration is chosen lift, it follows that $\Lambda(f, \Upsilon_{f}) =\A$ and that $\overline{\Upsilon}_{f} = 1_{\A}$.
    Given a delta lens $(g, \psi) \colon \A \to \B$, the underlying functor $g$ is a discrete opfibration if $\psi(a, gw) = w$ holds for each morphism $w \colon a \rightarrow a'$ in $\A$; in this case, $\psi = \Upsilon_{g}$. 
\end{example}
    
\begin{proposition}
\label{proposition:delta-lens-from-diagram}
    Given a commutative triangle of functors 
    \begin{equation}
    \label{equation:delta-lens-from-diagram}
        \begin{tikzcd}[column sep=small]
            \X && \A \\
            & \B
            \arrow["p", from=1-1, to=1-3]
            \arrow["fp"', from=1-1, to=2-2]
            \arrow["f", from=1-3, to=2-2]
        \end{tikzcd}
    \end{equation}
    such that $p$ is identity-on-objects and $fp$ is a discrete opfibration, there is a delta lens $(f, \varphi) \colon \A \to \B$ and a unique isomorphism $j \colon \X \cong \Lambda(f, \varphi)$ such that $\overline{\varphi} j = p$. 
\end{proposition}
\begin{proof}
    Since $fp \colon \X \to \B$ is a discrete opfibration, it admits a unique delta lens structure $(fp, \Upsilon_{fp}) \colon \X \to \B$ by Example~\ref{example:discrete-opfibration}.
    Define a delta lens $(f, \varphi) \colon \A \to \B$ whose lifting operation is given by $\varphi(a, u) = p \Upsilon_{fp}(a, u)$; this is well-defined by commutativity of \eqref{equation:delta-lens-from-diagram} and functoriality of $p$. 
    By Lemma~\ref{lemma:Lambda-category}, the category $\Lambda(f, \varphi)$ has the same objects as $\A$, thus the same objects as $\X$. 
    Recalling that every morphism in the domain of the discrete opfibration $fp$ is a chosen lift, there is an identity-on-objects functor $j \colon \X \to \Lambda(f, \varphi)$ which sends $\Upsilon_{fp}(a, u)$ to $\varphi(a, u)$. 
    The functor $j$ is clearly an isomorphism, since $p$ is faithful by commutativity of \eqref{equation:delta-lens-from-diagram}, and the equation $\overline{\varphi}j = p$ is satisfied by construction. 
\end{proof}

Together Lemma~\ref{lemma:Lambda-category} and Proposition~\ref{proposition:delta-lens-from-diagram} tell us that delta lenses are \emph{the same as} certain commutative diagrams of functors. 
This idea forms the cornerstone of many results concerning delta lenses, and is extended to an equivalence of categories in the following subsection. 

\subsection{The category of delta lenses}

Delta lenses may be considered as both objects and as morphisms, which suggests double categories are a suitable framework in which to study them \cite{Clarke2022}. 
In this paper, we focus on the category whose \emph{objects} are delta lenses, which we now recall. 

\begin{definition}
\label{definition:Lens-category}
    Let $\Lens$ be the category whose objects are delta lenses, and whose morphisms $(h, k) \colon (f, \varphi) \to (g, \psi)$, as depicted below, consist of a pair of functors such that $kf = gh$ and $h\varphi(a,u) = \psi(ha, ku)$. 
    \begin{equation}
    \label{equation:Lens-category}
        \begin{tikzcd}
            \A & \C \\
            \B & \D
            \arrow["h", from=1-1, to=1-2]
            \arrow["{(f, \varphi)}"', from=1-1, to=2-1]
            \arrow["{(g, \psi)}", from=1-2, to=2-2]
            \arrow["k"', from=2-1, to=2-2]
        \end{tikzcd}
    \end{equation}
\end{definition}

There is a faithful functor $\Lens \to \Cat^{\two}$ that sends each delta lens to its underlying functor. 
Thus, a morphism of delta lenses may be understood as a commutative square of functors that preserves the chosen lifts. 
Since each delta lens determines a wide subcategory of chosen lifts (Lemma~\ref{lemma:Lambda-category}), we obtain the following result. 

\begin{lemma}
\label{lemma:Lambda-morphism}
    Given a morphism \eqref{equation:Lens-category} of delta lenses, there exists a unique functor $\Lambda(h, k)$ such that the following diagram commutes. 
    \begin{equation}
    \label{equation:Lambda-morphism}
        \begin{tikzcd}
            {\Lambda(f,\varphi)} & {\Lambda(g,\psi)} \\
            \A & \B
            \arrow["{\Lambda(h, k)}", from=1-1, to=1-2]
            \arrow["{\overline{\varphi}}"', from=1-1, to=2-1]
            \arrow["{\overline{\psi}}", from=1-2, to=2-2]
            \arrow["h"', from=2-1, to=2-2]
        \end{tikzcd}
    \end{equation}
\end{lemma}
\begin{proof}
    Commutativity required that $\Lambda(h,k)$ has an assignment on objects $a \mapsto ha$ and an assignment on morphisms $\varphi(a, u) \mapsto \psi(ha,ku)$. 
\end{proof}

Let $\DOpf$ be the full subcategory of $\Cat^{\two}$ whose objects are discrete opfibrations. 
We now demonstrate that the relationship between discrete opfibrations and delta lenses extends to an adjunction.

\begin{proposition}
\label{proposition:DOpf-Lens}
    There is a coreflective adjunction
    \begin{equation*}
        \begin{tikzcd}
            \DOpf & \Lens
            \arrow[""{name=0, anchor=center, inner sep=0}, "\Upsilon"', shift right=2, hook, from=1-1, to=1-2]
            \arrow[""{name=1, anchor=center, inner sep=0}, "\Lambda"', shift right=2, from=1-2, to=1-1]
            \arrow["\dashv"{anchor=center, rotate=90}, draw=none, from=0, to=1]
        \end{tikzcd}
    \end{equation*}
    between the category of discrete opfibrations and the category of delta lenses. 
\end{proposition}
\begin{proof}
    The functor $\Upsilon \colon \DOpf \to \Lens$ sends each discrete opfibration $f$ to its unique delta lens structure $(f, \Upsilon_{f})$ by Example~\ref{example:discrete-opfibration}; it is clearly fully faithful. 
    The functor $\Lambda \colon \Lens \to \DOpf$ sends each delta lens $(f, \varphi)$ to the discrete opfibration $f \, \overline{\varphi}$ as described in Lemma~\ref{lemma:Lambda-category} and Lemma~\ref{lemma:Lambda-morphism}.

    Given a discrete opfibration $f \colon \A \to \B$, the wide subcategory of chosen lifts is $\Lambda(f, \Upsilon_{f}) = \A$, since $\Upsilon_{f}(a, fw) = w$ for each morphism $w \colon a \to a'$ in $\A$. 
    Therefore, the unit of the adjunction is the identity natural transformation. 
    The counit component of the adjunction at a delta lens $(f, \varphi) \colon \A \to \B$ is given by the following morphism of delta lenses. 
    \begin{equation*}
        \begin{tikzcd}
            {\Lambda(f, \varphi)} & \A \\
            \B & \B
            \arrow["{\overline{\varphi}}", from=1-1, to=1-2]
            \arrow["{(f \, \overline{\varphi}, \Upsilon_{f \, \overline{\varphi}})}"', from=1-1, to=2-1]
            \arrow["{(f, \varphi)}", from=1-2, to=2-2]
            \arrow[equals, from=2-1, to=2-2]
        \end{tikzcd}
    \end{equation*}
    Naturality holds by Lemma~\ref{lemma:Lambda-morphism}, and the triangle identities are easily verified. 
\end{proof}

\begin{remark}
\label{remark:DOpf-Lens}
    The functor $\Lambda \colon \Lens \to \DOpf$ admits a fully faithful right adjoint, which is constructed by factorising a discrete opfibration into an identity-on-objects functor followed by the fully faithful functor, the latter of which admits a canonical delta lens structure. 
    The functor $\Upsilon \colon \DOpf \to \Lens$ admits a left adjoint which is constructed by factorising the underlying functor of a delta lens into an initial functor followed by a discrete opfibration. 
\end{remark}

We want to construct an equivalence of categories between $\Lens$ and a category, whose objects are certain commutative diagrams of functors as in Proposition~\ref{proposition:delta-lens-from-diagram}, which we now introduce. 
The corresponding double category was considered in previous work \cite[Section~2]{Clarke2024}.

\begin{definition}
\label{definition:DiaLens}
    A \emph{diagrammatic delta lens} $(f, p)$ is a composable pair of functors $f \colon \A \to \B$ and $p \colon \X \to \A$ such that $p$ is identity-on-objects and $f \circ p$ is a discrete opfibration. 
    Let $\DiaLens$ be the category whose objects are diagrammatic delta lenses, and whose morphisms $(\ell, h, k) \colon (f, p) \to (g, q)$ are triples of functors such that the following diagram commutes. 
    \begin{equation}
    \label{equation:DiaLens}
        \begin{tikzcd}
            \X & \Y \\
            \A & \C \\
            \B & \D
            \arrow["\ell", from=1-1, to=1-2]
            \arrow["p"', from=1-1, to=2-1]
            \arrow["fp"', curve={height=24pt}, from=1-1, to=3-1]
            \arrow["q", from=1-2, to=2-2]
            \arrow["gq", curve={height=-24pt}, from=1-2, to=3-2]
            \arrow["h", from=2-1, to=2-2]
            \arrow["f"', from=2-1, to=3-1]
            \arrow["g", from=2-2, to=3-2]
            \arrow["k"', from=3-1, to=3-2]
        \end{tikzcd}
    \end{equation}
\end{definition}

Note that the functor $\ell \colon \X \to \Y$ in a morphism \eqref{equation:DiaLens} of diagrammatic delta lenses is unique by the properties of identity-on-objects functors and discrete opfibrations. 
Therefore, there is a faithful functor $\DiaLens \to \Cat^{\two}$ that sends each diagrammatic delta lens $(f, p)$ to $f$. 

\begin{lemma}
\label{lemma:Lens-to-DiaLens}
    There is a functor $\hat{\Lambda} \colon \Lens \to \DiaLens$ that assigns a delta lens $(f, \varphi)$ to the diagrammatic delta lens $(f, \overline{\varphi})$. 
\end{lemma}
\begin{proof}
    The functor $\hat{\Lambda}$ is well-defined on objects by Lemma~\ref{lemma:Lambda-category} and on morphisms by Lemma~\ref{lemma:Lambda-morphism}.
    Equivalently, we may invoke naturality of the counit of the adjunction $\Upsilon \dashv \Lambda$ in Proposition~\ref{proposition:DOpf-Lens}.
\end{proof}

\begin{lemma}
\label{lemma:DiaLens-to-Lens}
    There is a functor $\hat{\Upsilon} \colon \DiaLens \to \Lens$ that assigns a diagrammatic delta lens $(f, p)$ to the delta lens $(f, p\Upsilon_{fp})$.  
\end{lemma}
\begin{proof}
    The functor $\hat{\Upsilon}$ is well-defined on objects by Proposition~\ref{proposition:delta-lens-from-diagram}. 
    Given a morphism \eqref{equation:DiaLens} of diagrammatic delta lenses, there is a morphism $(h, k) \colon (f, p \Upsilon_{fp}) \to (g, q \Upsilon_{gq})$ of delta lenses, since $kf = gh$ and, for all $a \in \A$ and $u \colon fa \to b$ in $\B$, the equation $hp \Upsilon_{fp}(a, u) = q \ell \Upsilon_{fp}(a, u) = q \Upsilon_{gq}(ha, ku)$ holds by functoriality of $\Upsilon$. 
\end{proof}

We are now ready to demonstrate that the category of delta lenses is equivalent to the category of diagrammatic delta lenses; one might call this a ``representation theorem'' for delta lenses. 
This equivalence has been used implicitly for years \cite{Clarke2020} however a detailed proof in the literature remained absent; we now fill this gap. 

\begin{theorem}
\label{theorem:Lens-DiaLens}
    There is an equivalence of categories $\Lens \simeq \DiaLens$. 
\end{theorem}
\begin{proof}
    Given functors $\hat{\Lambda} \colon \Lens \to \DiaLens$ and $\hat{\Upsilon} \colon \DiaLens \to \Lens$ defined in Lemma~\ref{lemma:Lens-to-DiaLens} and Lemma~\ref{lemma:DiaLens-to-Lens}, respectively, we show that there are natural isomorphisms $\hat{\Upsilon} \circ \hat{\Lambda} \cong 1$ and $\hat{\Lambda} \circ \hat{\Upsilon} \cong 1$. 
    We construct the components of these isomorphisms, omitting the straightforward proof of naturality. 

    Applying the functor $\hat{\Upsilon} \circ \hat{\Lambda}$ to a delta lens $(f, \varphi) \colon \A \to \B$ we obtain a delta lens $(f, \overline{\varphi} \, \Upsilon_{f \, \overline{\varphi}}) \colon \A \to \B$. 
    Since $\Upsilon_{f \, \overline{\varphi}}(a, u) = \varphi(a, u)$ by construction, and $\overline{\varphi}$ is a wide subcategory inclusion, it follows that $\overline{\varphi} \, \Upsilon_{f \, \overline{\varphi}}(a, u) = \varphi(a, u)$ for all $a \in \A$ and $u \colon fa \to b$ in $\B$. 
    Therefore, we have an equality of functors $\hat{\Upsilon} \circ \hat{\Lambda} = 1$. 

    Applying the functor $\hat{\Lambda} \circ \hat{\Upsilon}$ to a diagrammatic delta lens $(f, p)$ we obtain a diagrammatic delta lens $(f, \overline{\varphi})$ where $\varphi = \overline{p\Upsilon}_{fp}$. 
    To show that $1 \cong \hat{\Lambda} \circ \hat{\Upsilon}$, we need an isomorphism $j \colon \X \cong \Lambda(f, \varphi)$ such that $\overline{\varphi}j = p$. 
    However, this follows immediately from Proposition~\ref{proposition:delta-lens-from-diagram}, completing the proof. 
\end{proof}

\subsection{Split opfibrations as delta lenses}
\label{sec:split-opfibration}

Discrete opfibrations are a special case of a more general notion called a split opfibration. 
Although split opfibrations are typically defined as opfibrations equipped with certain \emph{additional structure}, we instead define them as delta lenses satisfying a certain \emph{property}.

\begin{example}
\label{example:split-opfibration}
    A \emph{split opfibration} is a delta lens $(f, \varphi) \colon \A \rightarrow \B$ such that each chosen lift $\varphi(a, u)$ is \emph{opcartesian}. 
    This means that for each morphism $w \colon a \rightarrow a''$ in $\A$ and composable pair of morphisms in $\B$ such that $fw = v \circ u$, there exists a unique morphism $w' \colon a' \rightarrow a''$ in $\A$ such that $w' \circ \varphi(a, u) = w$ and $fw' = v$. 
    \begin{equation*}
        \begin{tikzcd}
            \A 
            \arrow[d, "{(f, \, \varphi)}"']
            & 
            a 
            \arrow[r, "{\varphi(a, \, u)}"']
            \arrow[d, phantom, "\vdots"]
            \arrow[rounded corners, to path={-- ([yshift=2.2ex]\tikztostart.north) -- node[above]{\scriptsize$w$}  ([yshift=1.5ex]\tikztotarget.north) -- (\tikztotarget)}]{rr}
            &
            a'
            \arrow[r, dashed, "\exists! \, w'"']
            \arrow[d, phantom, "\vdots"]
            &
            a''
            \arrow[d, phantom, "\vdots"]
            \\
            \B 
            &
            fa
            \arrow[r, "u"]
            \arrow[rounded corners, to path={-- ([yshift=-1.5ex]\tikztostart.south) -- node[below]{\scriptsize$fw$}  ([yshift=-1.5ex]\tikztotarget.south) -- (\tikztotarget)}]{rr}
            &
            b 
            \arrow[r, "v"]
            & 
            fa''
        \end{tikzcd}
    \end{equation*}

    Equivalently, a split opfibration is a delta lens such that each chosen lift is \emph{weakly opcartesian}. 
    This means that for each morphism $w \colon a \rightarrow a''$ in $\A$, there exists a unique morphism $\tilde{w} \colon a' \rightarrow a''$ in $\A$ such that $\tilde{w} \circ \varphi(a, u) = w$ and $f\hat{w} = 1$. 
    \begin{equation*}
        \begin{tikzcd}
            &
            & 
            a'
            \arrow[d, dashed, "\exists! \, \tilde{w}"]
            \\[-1ex]
            \A 
            \arrow[d, "{(f, \, \varphi)}"']
            & 
            a 
            \arrow[r, "w"']
            \arrow[d, phantom, "\vdots"]
            \arrow[ru, "{\varphi(a, \, u)}"]
            &
            a''
            \arrow[d, phantom, "\vdots"]
            \\
            \B 
            &
            fa
            \arrow[r, "fw"]
            &
            fa''
        \end{tikzcd}
    \end{equation*}
    Since chosen lifts are closed under composition by axiom \ref{DL3}, it is straightforward to show that these characterisations of a split opfibration in terms of opcartesian and weakly opcartesian lifts are equivalent. 
\end{example}

Let $\SOpf$ be the full subcategory of $\Lens$ whose objects are split opfibrations. 
It is natural to wonder if we can characterise the essential image of the subcategory inclusion $\SOpf \to \Lens$ followed by the equivalence $\Lens \simeq \DiaLens$. 

Recall that there is an endofunctor $\Dec \colon \Cat \to \Cat$, which sends each category $\A$ to the coproduct $\sum_{a \in \A} \A / a$ of its slice categories, called the \emph{décalage} of $\A$. 
This endofunctor is copointed, and the component of $\varepsilon \colon \Dec \Rightarrow 1$ at a category $\A$ is the projection $\dom \colon \sum_{a \in \A} \A / a \to \A$ which sends each morphism to its domain. 

\begin{proposition}[{\cite[Theorem~4.7]{Clarke2020b}}]
\label{proposition:diagrammatic-SOpf}
    A diagrammatic delta lens $(f, p)$ corresponds to a split opfibration under the equivalence $\DiaLens \simeq \Lens$ if and only if the functor $\Dec(f) \circ \pi_{2}$ is a discrete opfibration. 
    \begin{equation*}
        \begin{tikzcd}
            {\X \times_{\A}\Dec(\A)} & {\Dec(\A)} & {\Dec(\B)} \\
            \X & \A & \B
            \arrow["{\pi_{2}}", from=1-1, to=1-2]
            \arrow["{\pi_{1}}"', from=1-1, to=2-1]
            \arrow["\big\lrcorner"{anchor=center, pos=0.125}, draw=none, from=1-1, to=2-2]
            \arrow["{\Dec(f)}", from=1-2, to=1-3]
            \arrow["{\varepsilon_{\A}}", from=1-2, to=2-2]
            \arrow["{\varepsilon_{\B}}", from=1-3, to=2-3]
            \arrow["p"', from=2-1, to=2-2]
            \arrow["f"', from=2-2, to=2-3]
        \end{tikzcd}
    \end{equation*}
\end{proposition}

%% file: double-categories.tex
\section{The double category of split multivalued functions}
\label{sec:SMULT}

In this section, we introduce the double category $\SMULT$ of sets, functions, and split multivalued functions. 
We also introduce the notion of \emph{globular transformation} between double functors, akin to an icon of $2$-functors. 
Our main results are that the $\SMULT$ admits a $1$-dimensional (Theorem~\ref{theorem:1d-universal-property-SMULT}) and $2$-dimensional (Theorem~\ref{theorem:2d-universal-property-SMULT}) universal property as the ``terminal'' globular cone over $K_{\ast} \colon \SQ(\Set) \to \SPAN$. 
We obtain an isomorphism $\Idx(\Cat, \SMULT) \cong \GlobCone(\Cat, K_{\ast})$ of categories whose objects are lax double functors $\LO(\B) \to \SMULT$ and globular cones over $K_{\ast}$, respectively (Corollary~\ref{corollary:Idx-GlobCone}).  

We assume familiarity with basic double category theory \cite{Grandis2019}. 
To establish our conventions, a double category consists of four kinds of things: \emph{objects}, \emph{tight morphisms} $A \to B$ (drawn vertically), \emph{loose morphisms} $A \lto B$ (drawn horizontally), and \emph{cells}. 
We have adopted the terminology of \emph{tight} and \emph{loose} from $\mathcal{F}$-categories~\cite{LackShulman2012}.
Composition in the tight direction is strict, while composition in the loose direction is associative and unital up to specified isomorphism. 
A double category is called \emph{strict} if composition in both the tight and loose directions is strict.

\subsection{Split multivalued functions}

\begin{definition}
\label{definition:split-multivalued-function}
    A \emph{multivalued function} $A \lto B$ is a span of functions
    \begin{equation}
    \label{equation:multivalued-function}
        \begin{tikzcd}
            A
            & 
            X 
            \arrow[l, two heads, "s"']
            \arrow[r, "t"]
            & 
            B
        \end{tikzcd}
    \end{equation}
    such that $s$ is an epimorphism.
    A \emph{split multivalued function} $A \lto B$ is a diagram of functions 
    \begin{equation}
    \label{equation:split-mutlivalued-function}
        \begin{tikzcd}
            A
            & 
            X 
            \arrow[l, two heads, "s", shift left = 2]
            \arrow[from=l, tail, "\sigma", shift left = 2]
            \arrow[r, "t"]
            & 
            B
        \end{tikzcd}
    \end{equation}
    such that $s \circ \sigma = 1_{A}$. 
\end{definition}

Given a multivalued function \eqref{equation:multivalued-function}, the fibre $s^{-1}\{a\}$ of the function $s \colon X \to A$ at each element $a \in A$ is a non-empty subset of $X$, and the restriction of the function $t \colon X \to B$ to this subset determines the ``multiple values'' of the element $a \in A$. 
For a split multivalued function \eqref{equation:split-mutlivalued-function}, each fibre of $s \colon X \to A$ is equipped with a chosen element $\sigma(a) \in s^{-1}\{a\}$, which is furthermore sent to an element $t(\sigma(a)) \in B$. 

There are various restrictions we could place of the definition of (split) multivalued function, such as working with isomorphism classes of spans, asking that the epimorphism leg have finite fibres, or requiring underlying span to be jointly monic. 
Our particular choice of definition provides the required level of generality to prove our desired results about delta lenses.

\begin{definition}
\label{definition:double-category-SMult}
    The \emph{double category of split multivalued functions}, denoted $\SMULT$, consists of the following data. 
    \begin{enumerate}[label=(\alph*)]
        \item Objects are sets. 
        \item Tight morphisms $f \colon A \to B$ are functions. 
        \item Loose morphisms $(s, X, t, \sigma) \colon A \lto B$ are split multivalued functions \eqref{equation:split-mutlivalued-function}.
        \item Cells correspond to diagrams of functions, as depicted below, such that the equations $q_{2} \circ \alpha = g \circ p_{2}$, $q_{1} \circ \alpha = f \circ p_{1}$, and $\alpha \circ \varphi = \psi \circ f$ hold. 
        \begin{equation}
        \label{equation:SMULT-cell}
            \begin{tikzcd}[column sep = large]
                A & B \\
                C & D
                \arrow["{(p_{1},\, X,\, p_{2},\, \varphi)}"{inner sep = 1ex}, "\shortmid"{marking}, from=1-1, to=1-2]
                \arrow[""{name=0, anchor=center, inner sep=0}, "f"', from=1-1, to=2-1]
                \arrow[""{name=1, anchor=center, inner sep=0}, "g", from=1-2, to=2-2]
                \arrow["{(q_{1},\, Y,\, q_{2},\, \psi)}"'{inner sep = 1ex}, "\shortmid"{marking}, from=2-1, to=2-2]
                \arrow["\alpha"{description}, draw=none, from=0, to=1]
            \end{tikzcd}
            \quad \coloneqq \quad 
            \begin{tikzcd}
                A & X & B \\
                C & Y & D
                \arrow["\varphi", shift left=2, tail, from=1-1, to=1-2]
                \arrow["f"', from=1-1, to=2-1]
                \arrow["{p_{1}}", shift left=2, two heads, from=1-2, to=1-1]
                \arrow["{p_{2}}", from=1-2, to=1-3]
                \arrow["\alpha", from=1-2, to=2-2]
                \arrow["g", from=1-3, to=2-3]
                \arrow["\psi", shift left=2, tail, from=2-1, to=2-2]
                \arrow["{q_{1}}", shift left=2, two heads, from=2-2, to=2-1]
                \arrow["{q_{2}}"', from=2-2, to=2-3]
            \end{tikzcd}
        \end{equation}
        \item Composition of tight morphisms is composition of functions, and identity tight morphisms are identity functions. 
        \item Given loose morphisms $(p_{1}, X, p_{2}, \varphi) \colon A \lto B$ and $(q_{1}, Y, q_{2}, \psi) \colon B \lto C$, their composite is given by the following diagram of functions, where $Z$ is the pullback of $p_{2}$ and $q_{1}$ with corresponding projections $\pi_{X}$ and $\pi_{Y}$, and $\langle 1_{X}, \psi \circ p_{2} \rangle \colon X \to Z$ is the uniquely induced morphism into the pullback. 
        \begin{equation*}
            \begin{tikzcd}
                && Z \\
                & X && Y \\
                A && B && C
                \arrow["{\pi_{X}}", shift left=2, two heads, from=1-3, to=2-2]
                \arrow["{\pi_{Y}}", from=1-3, to=2-4]
                \arrow["\lrcorner"{anchor=center, pos=0.125, rotate=-45}, draw=none, from=1-3, to=3-3]
                \arrow["{\langle 1_{X},\, \psi \circ p_{2} \rangle}", shift left=2, tail, from=2-2, to=1-3]
                \arrow["{p_{1}}", shift left=2, two heads, from=2-2, to=3-1]
                \arrow["{p_{2}}", from=2-2, to=3-3]
                \arrow["{q_{1}}", shift left=2, two heads, from=2-4, to=3-3]
                \arrow["{q_{2}}", from=2-4, to=3-5]
                \arrow["\varphi", shift left=2, tail, from=3-1, to=2-2]
                \arrow["\psi", shift left=2, tail, from=3-3, to=2-4]
            \end{tikzcd}
        \end{equation*}
        The identity loose morphism on a set $A$ is given by $(1_{A}, A, 1_{A}, 1_{A}) \colon A \lto A$. 
        \item The identity cells on a tight morphism $f \colon A \to B$ and a loose morphism $(s, X, t, \sigma) \colon A \lto B$ are given by the following diagrams, respectively. 
        \begin{equation*}
            \begin{tikzcd}
                A & A & A \\
                B & B & B
                \arrow["{1_{A}}", shift left=2, tail, from=1-1, to=1-2]
                \arrow["f"', from=1-1, to=2-1]
                \arrow["{1_{A}}", shift left=2, two heads, from=1-2, to=1-1]
                \arrow["{1_{A}}", from=1-2, to=1-3]
                \arrow["f", from=1-2, to=2-2]
                \arrow["f", from=1-3, to=2-3]
                \arrow["{1_{B}}", shift left=2, tail, from=2-1, to=2-2]
                \arrow["{1_{B}}", shift left=2, two heads, from=2-2, to=2-1]
                \arrow["{1_{B}}"', from=2-2, to=2-3]
            \end{tikzcd}
            \qquad \qquad 
            \begin{tikzcd}
                A & X & B \\
                A & X & B
                \arrow["\sigma", shift left=2, tail, from=1-1, to=1-2]
                \arrow["{1_{A}}"', from=1-1, to=2-1]
                \arrow["s", shift left=2, two heads, from=1-2, to=1-1]
                \arrow["t", from=1-2, to=1-3]
                \arrow["{1_{X}}", from=1-2, to=2-2]
                \arrow["{1_{B}}", from=1-3, to=2-3]
                \arrow["\sigma", shift left=2, tail, from=2-1, to=2-2]
                \arrow["s", shift left=2, two heads, from=2-2, to=2-1]
                \arrow["t"', from=2-2, to=2-3]
            \end{tikzcd}
        \end{equation*}
        \item The composite of cells in the tight direction is given by the following. 
        \begin{equation*}
            \begin{tikzcd}
                A & X & B \\
                {A'} & {X'} & {B'} \\
                {A''} & {X''} & {B''}
                \arrow["\sigma", shift left=2, tail, from=1-1, to=1-2]
                \arrow["f"', from=1-1, to=2-1]
                \arrow["s", shift left=2, two heads, from=1-2, to=1-1]
                \arrow["t", from=1-2, to=1-3]
                \arrow["\alpha", from=1-2, to=2-2]
                \arrow["g", from=1-3, to=2-3]
                \arrow["{\sigma'}", shift left=2, tail, from=2-1, to=2-2]
                \arrow["h"', from=2-1, to=3-1]
                \arrow["{s'}", shift left=2, two heads, from=2-2, to=2-1]
                \arrow["{t'}"', from=2-2, to=2-3]
                \arrow["\beta", from=2-2, to=3-2]
                \arrow["k", from=2-3, to=3-3]
                \arrow["{\sigma''}", shift left=2, tail, from=3-1, to=3-2]
                \arrow["{s''}", shift left=2, two heads, from=3-2, to=3-1]
                \arrow["{t''}"', from=3-2, to=3-3]
            \end{tikzcd}
            \quad = \quad 
            \begin{tikzcd}
                A & X & B \\
                {A''} & {X''} & {B''}
                \arrow["\sigma", shift left=2, tail, from=1-1, to=1-2]
                \arrow["hf"', from=1-1, to=2-1]
                \arrow["s", shift left=2, two heads, from=1-2, to=1-1]
                \arrow["t", from=1-2, to=1-3]
                \arrow["{\beta\alpha}", from=1-2, to=2-2]
                \arrow["kg", from=1-3, to=2-3]
                \arrow["{\sigma''}", shift left=2, tail, from=2-1, to=2-2]
                \arrow["{s''}", shift left=2, two heads, from=2-2, to=2-1]
                \arrow["{t''}"', from=2-2, to=2-3]
            \end{tikzcd}
        \end{equation*}
        \item The composite of cells in the loose direction is given by the following diagram where some labels have been omitted, and the function $\alpha \times \beta$ is the uniquely induced morphism into the pullback of $p_{2}'$ and $q_{1}'$ determined by the equation $p_{2}' \circ \alpha \circ \pi_{X} = q_{1}' \circ \beta \circ \pi_{Y}$. 
        \begin{equation*}
            \begin{tikzcd}
                A & X & B & Y & C \\
                {A'} & {X'} & {B'} & {Y'} & {C'}
                \arrow["\varphi", shift left=2, tail, from=1-1, to=1-2]
                \arrow["f"', from=1-1, to=2-1]
                \arrow["{p_{1}}", shift left=2, two heads, from=1-2, to=1-1]
                \arrow["{p_{2}}", from=1-2, to=1-3]
                \arrow["\alpha", from=1-2, to=2-2]
                \arrow["\psi", shift left=2, tail, from=1-3, to=1-4]
                \arrow["g", from=1-3, to=2-3]
                \arrow["{q_{1}}", shift left=2, two heads, from=1-4, to=1-3]
                \arrow["{q_{2}}", from=1-4, to=1-5]
                \arrow["\beta", from=1-4, to=2-4]
                \arrow["h", from=1-5, to=2-5]
                \arrow["{\varphi'}", shift left=2, tail, from=2-1, to=2-2]
                \arrow["{p_{1}'}", shift left=2, two heads, from=2-2, to=2-1]
                \arrow["{p_{2}'}"', from=2-2, to=2-3]
                \arrow["{\psi'}", shift left=2, tail, from=2-3, to=2-4]
                \arrow["{q_{1}'}", shift left=2, two heads, from=2-4, to=2-3]
                \arrow["{q_{2}'}"', from=2-4, to=2-5]
            \end{tikzcd}
            \quad = \quad 
            \begin{tikzcd}[column sep = 2em]
                A & {X\times_{B}Y} & C \\
                {A'} & {X'\times_{B'}Y'} & {C'}
                \arrow[shift left=2, tail, from=1-1, to=1-2]
                \arrow["f"', from=1-1, to=2-1]
                \arrow["{p_{1}\pi_{X}}"{inner sep = 1ex}, shift left=2, two heads, from=1-2, to=1-1]
                \arrow["{q_{2}\pi_{Y}}", from=1-2, to=1-3]
                \arrow["{\alpha \times \beta}", from=1-2, to=2-2]
                \arrow["h", from=1-3, to=2-3]
                \arrow[shift left=2, tail, from=2-1, to=2-2]
                \arrow["{p_{1}'\pi_{X'}}"{inner sep = 1ex}, shift left=2, two heads, from=2-2, to=2-1]
                \arrow["{q_{2}'\pi_{Y'}}"', from=2-2, to=2-3]
            \end{tikzcd}
        \end{equation*}
        \item If we choose pullbacks in $\Set$ such that the pullback of the identity function is an identity function, then composition of loose morphisms is strictly unital. The associator isomorphisms are determined by the natural isomorphisms $(X \times_{B} Y) \times_{C} Z \cong X \times_{B} (Y \times_{C} Z)$ of pullbacks, however we omit the details. 
    \end{enumerate}
\end{definition}

While we have stated the definition of $\SMULT$ explicitly, it is also natural to ask if this double category admits a characterisation via some universal property. 
The rest of this section is dedicated to providing such a characterisation.
We begin by recalling two important double categories which are closely related to $\SMULT$. 

\begin{example}
\label{example:double-category-SPAN}
    The \emph{double category of spans}, denoted $\SPAN$, is the double category whose objects are sets, whose tight morphisms are functions, whose loose morphisms are spans, and whose cells correspond to diagrams of functions, as depicted below, such that $q_{1} \circ \alpha = f \circ p_{1}$ and $q_{2} \circ \alpha = g \circ p_{2}$. 
    \begin{equation*}
        \begin{tikzcd}[column sep = large]
            A & B \\
            C & D
            \arrow["{(p_{1},\, X,\, p_{2})}"{inner sep = 1ex}, "\shortmid"{marking}, from=1-1, to=1-2]
            \arrow[""{name=0, anchor=center, inner sep=0}, "f"', from=1-1, to=2-1]
            \arrow[""{name=1, anchor=center, inner sep=0}, "g", from=1-2, to=2-2]
            \arrow["{(q_{1},\, Y,\, q_{2})}"'{inner sep = 1ex}, "\shortmid"{marking}, from=2-1, to=2-2]
            \arrow["\alpha"{description}, draw=none, from=0, to=1]
        \end{tikzcd}
        \quad \coloneqq \quad 
        \begin{tikzcd}
            A & X & B \\
            C & Y & D
            \arrow["f"', from=1-1, to=2-1]
            \arrow["{p_{1}}"', from=1-2, to=1-1]
            \arrow["{p_{2}}", from=1-2, to=1-3]
            \arrow["\alpha", from=1-2, to=2-2]
            \arrow["g", from=1-3, to=2-3]
            \arrow["{q_{1}}", from=2-2, to=2-1]
            \arrow["{q_{2}}"', from=2-2, to=2-3]
        \end{tikzcd}
    \end{equation*}   
    There is a strict double functor $U_{2} \colon \SMULT \rightarrow \SPAN$ which is determined by the following assignment on cells. 
    \begin{equation}
    \label{equation:SMULT-to-SPAN}
        \begin{tikzcd}
            A & X & B \\
            C & Y & D
            \arrow["\varphi", shift left=2, tail, from=1-1, to=1-2]
            \arrow["f"', from=1-1, to=2-1]
            \arrow["{p_{1}}", shift left=2, two heads, from=1-2, to=1-1]
            \arrow["{p_{2}}", from=1-2, to=1-3]
            \arrow["\alpha", from=1-2, to=2-2]
            \arrow["g", from=1-3, to=2-3]
            \arrow["\psi", shift left=2, tail, from=2-1, to=2-2]
            \arrow["{q_{1}}", shift left=2, two heads, from=2-2, to=2-1]
            \arrow["{q_{2}}"', from=2-2, to=2-3]
        \end{tikzcd}
        \quad \longmapsto \quad
        \begin{tikzcd}
            A & X & B \\
            C & Y & D
            \arrow["f"', from=1-1, to=2-1]
            \arrow["{p_{1}}"', from=1-2, to=1-1]
            \arrow["{p_{2}}", from=1-2, to=1-3]
            \arrow["\alpha", from=1-2, to=2-2]
            \arrow["g", from=1-3, to=2-3]
            \arrow["{q_{1}}", from=2-2, to=2-1]
            \arrow["{q_{2}}"', from=2-2, to=2-3]
        \end{tikzcd}
    \end{equation}
\end{example}

\begin{example}
\label{example:double-category-SQSET}
    The \emph{double category of squares in $\Set$}, denoted $\SQ(\Set)$, is the double category whose objects are sets, whose tight and loose morphisms are both functions, and whose cells correspond to commuting squares of functions; this is a strict double category. 
    There is a strict double functor $U_{1} \colon \SMULT \to \SQ(\Set)$ which is determined by the following assignment on cells. 
    \begin{equation}
    \label{equation:SMULT-to_SQSET}
        \begin{tikzcd}
            A & X & B \\
            C & Y & D
            \arrow["\varphi", shift left=2, tail, from=1-1, to=1-2]
            \arrow["f"', from=1-1, to=2-1]
            \arrow["{p_{1}}", shift left=2, two heads, from=1-2, to=1-1]
            \arrow["{p_{2}}", from=1-2, to=1-3]
            \arrow["\alpha", from=1-2, to=2-2]
            \arrow["g", from=1-3, to=2-3]
            \arrow["\psi", shift left=2, tail, from=2-1, to=2-2]
            \arrow["{q_{1}}", shift left=2, two heads, from=2-2, to=2-1]
            \arrow["{q_{2}}"', from=2-2, to=2-3]
        \end{tikzcd}
        \quad \longmapsto \quad
        \begin{tikzcd}
            A & B \\
            C & D
            \arrow["p_{2} \circ \varphi", from=1-1, to=1-2]
            \arrow["f"', from=1-1, to=2-1]
            \arrow["g", from=1-2, to=2-2]
            \arrow["q_{2} \circ \psi"', from=2-1, to=2-2]
        \end{tikzcd}
    \end{equation}
\end{example}

The double categories $\SPAN$ and $\SQ(\Set)$ related by a strict double functor $K_{\ast} \colon \SQ(\Set) \rightarrow \SPAN$ with the following assignment on cells. 
    \begin{equation}
    \label{equation:SQSET-to-SPAN}
        \begin{tikzcd}
            A & B \\
            C & D
            \arrow["h", from=1-1, to=1-2]
            \arrow["f"', from=1-1, to=2-1]
            \arrow["g", from=1-2, to=2-2]
            \arrow["k"', from=2-1, to=2-2]
        \end{tikzcd}
        \quad \longmapsto \quad
        \begin{tikzcd}
            A & A & B \\
            C & C & D
            \arrow["f"', from=1-1, to=2-1]
            \arrow["{1_{A}}"', from=1-2, to=1-1]
            \arrow["h", from=1-2, to=1-3]
            \arrow["f", from=1-2, to=2-2]
            \arrow["g", from=1-3, to=2-3]
            \arrow["{1_{C}}", from=2-2, to=2-1]
            \arrow["k"', from=2-2, to=2-3]
        \end{tikzcd}
    \end{equation}
The notation $K_{\ast}$ for this double functor is to remind the reader that it provides a choice of \emph{companions} \cite{DawsonParePronk2010, GrandisPare2004} for the double category $\SPAN$.
To state the precise relationship between the double functors \eqref{equation:SMULT-to-SPAN}, \eqref{equation:SMULT-to_SQSET}, and \eqref{equation:SQSET-to-SPAN}, we must introduce a new definition. 

\subsection{Globular transformations}

Recall that a (tight) transformation $\tau \colon F \Rightarrow G$ between double functors $F, G \colon \AA \rightarrow \XX$ provides, for each object $A \in \AA$, a tight morphism $\tau_{A} \colon FA \to GA$ in $\XX$ and, for each loose morphism $p \colon A \lto B$ in $\AA$, a cell $\tau_{p}$ in $\XX$ as follows. 
    \begin{equation*}
        \begin{tikzcd}
            FA & FB \\
            GA & GB
            \arrow["Fp"{inner sep = 1ex}, "\shortmid"{marking}, from=1-1, to=1-2]
            \arrow[""{name=0, anchor=center, inner sep=0}, "{\tau_{A}}"', from=1-1, to=2-1]
            \arrow[""{name=1, anchor=center, inner sep=0}, "{\tau_{B}}", from=1-2, to=2-2]
            \arrow["Gp"'{inner sep = 1ex}, "\shortmid"{marking}, from=2-1, to=2-2]
            \arrow["{\tau_{p}}"{description}, draw=none, from=0, to=1]
        \end{tikzcd}
    \end{equation*}
We wish to consider transformations with a certain property which is analogous to that of an \emph{icon} between functors of bicategories \cite{Lack2010}. 

\begin{definition}
\label{definition:globular-transformation}
    Given (lax) double functors $F, G \colon \AA \rightarrow \XX$, a (tight) transformation $\tau \colon F \Rightarrow G$ is called \emph{globular} if the tight morphism $\tau_{A} \colon FA \rightarrow GA$ is an identity for each object $A \in \AA$,
\end{definition}

A globular transformation $\tau \colon F \Rightarrow G$ implies that $F$ and $G$ have the same assignment on objects, and by naturality, the same assignment on tight morphisms. 
The term \emph{globular} is chosen as the components of $\tau$ at a loose morphism are globular cells in the double category. 
Globular transformations are closed under composition and whiskering. 

\begin{lemma}
\label{lemma:globular-SMULT-SPAN-SQSET}
    There is a globular transformation between double functors as follows. 
    \begin{equation}
    \label{equation:SMULT-globular-transformation}
        \begin{tikzcd}[column sep = small]
            & \SMULT \\
            {\SQ(\Set)} && \SPAN
            \arrow[""{name=0, anchor=center, inner sep=0}, "{U_{1}}"', from=1-2, to=2-1]
            \arrow[""{name=1, anchor=center, inner sep=0}, "{U_{2}}", from=1-2, to=2-3]
            \arrow["{K_{\ast}}"', from=2-1, to=2-3]
            \arrow["\sigma", shift right=2, shorten <=10pt, shorten >=10pt, Rightarrow, from=0, to=1]
        \end{tikzcd}
    \end{equation}
\end{lemma}
\begin{proof}
    We define the globular transformation $\sigma \colon K_{\ast}U_{1} \Rightarrow U_{2}$ as follows. 
    For each object $A$ in $\SMULT$, we have the identity function $1_{A} \colon A \to A$ in $\SPAN$, and for each loose morphism \eqref{equation:split-mutlivalued-function} in $\SMULT$, we have the following globular cell in $\SPAN$. 
    \begin{equation*}
        \begin{tikzcd}
            A & A & B \\
            A & X & B
            \arrow["{1_{A}}"', from=1-1, to=2-1]
            \arrow["{1_{A}}"', from=1-2, to=1-1]
            \arrow["{t \sigma}", from=1-2, to=1-3]
            \arrow["\sigma", from=1-2, to=2-2]
            \arrow["{1_{B}}", from=1-3, to=2-3]
            \arrow["s", from=2-2, to=2-1]
            \arrow["t"', from=2-2, to=2-3]
        \end{tikzcd}
    \end{equation*}
    It is straightforward to verify that $\sigma \colon K_{\ast}U_{1} \Rightarrow U_{2}$ is natural with respect to the cells in $\SMULT$.
\end{proof}

\subsection{The universal property of the double category of split multivalued functions}

We now show that the globular transformation in Lemma~\ref{lemma:globular-SMULT-SPAN-SQSET} may be understood as universal globular cone over the double functor $K_{\ast} \colon \SQ(\Set) \to \SPAN$, and characterises the double category $\SMULT$ as a certain kind of limit. 

\begin{theorem}[1-dimensional universal property of $\SMULT$]
\label{theorem:1d-universal-property-SMULT}
    Given lax double functors $F_{1} \colon \BB \to \SQ(\Set)$ and $F_{2} \colon \BB \to \SPAN$ and a globular transformation 
    \begin{equation*}
        \begin{tikzcd}[column sep = small]
            & \BB \\
            {\SQ(\Set)} && \SPAN
            \arrow[""{name=0, anchor=center, inner sep=0}, "{F_{1}}"', from=1-2, to=2-1]
            \arrow[""{name=0p, anchor=center, inner sep=0}, phantom, from=1-2, to=2-1, start anchor=center, end anchor=center]
            \arrow[""{name=1, anchor=center, inner sep=0}, "{F_{2}}", from=1-2, to=2-3]
            \arrow[""{name=1p, anchor=center, inner sep=0}, phantom, from=1-2, to=2-3, start anchor=center, end anchor=center]
            \arrow["{{K_{\ast}}}"', from=2-1, to=2-3]
            \arrow["\varphi", shift right=2, shorten <=8pt, shorten >=8pt, Rightarrow, from=0p, to=1p]
        \end{tikzcd}
    \end{equation*}
    there exists a unique lax double functor $F \colon \BB \to \SMULT$ such that $\sigma \cdot F = \varphi$, where $\sigma \colon K_{\ast} U_{1} \Rightarrow U_{2}$ is defined in Lemma~\ref{lemma:globular-SMULT-SPAN-SQSET}. 
\end{theorem}
\begin{proof}
    Since $\varphi$ is a globular transformation, $F_{1}$ and $F_{2}$ agree on objects and tight morphisms; we denote this shared action by $F$. 
    Given a loose morphism $u \colon A \lto B$ in $\BB$, we denote the component of $\varphi$ at $u$ by the following cell in $\SPAN$. 
    \begin{equation*}
        \begin{tikzcd}
            FA & FB \\
            FA & FB
            \arrow["{F_{1}u}"{inner sep = 1ex}, "\shortmid"{marking}, from=1-1, to=1-2]
            \arrow[""{name=0, anchor=center, inner sep=0}, Rightarrow, no head, from=1-1, to=2-1]
            \arrow[""{name=1, anchor=center, inner sep=0}, Rightarrow, no head, from=1-2, to=2-2]
            \arrow["{F_{2}u}"'{inner sep = 1ex}, "\shortmid"{marking}, from=2-1, to=2-2]
            \arrow["{\varphi u}"{description}, draw=none, from=0, to=1]
        \end{tikzcd}
        \quad \coloneqq \quad 
        \begin{tikzcd}
            FA & FA & FB \\
            FA & {F(u)} & FB
            \arrow[Rightarrow, no head, from=1-1, to=2-1]
            \arrow[Rightarrow, no head, from=1-2, to=1-1]
            \arrow["{t_{u}\varphi_{u}}", from=1-2, to=1-3]
            \arrow["{\varphi_{u}}", from=1-2, to=2-2]
            \arrow[Rightarrow, no head, from=1-3, to=2-3]
            \arrow["{s_{u}}", from=2-2, to=2-1]
            \arrow["{t_{u}}"', from=2-2, to=2-3]
        \end{tikzcd}
    \end{equation*}
    This defines a split multivalued function $(s_{u}, F(u), t_{u}, \varphi_{u}) \colon FA \lto FB$ for each loose morphism $u \colon A \lto B$. 
    Given a cell 
    \begin{equation*}
        \begin{tikzcd}
            A & B \\
            C & D
            \arrow["u"{inner sep = 1ex}, "\shortmid"{marking}, from=1-1, to=1-2]
            \arrow[""{name=0, anchor=center, inner sep=0}, "f"', from=1-1, to=2-1]
            \arrow[""{name=1, anchor=center, inner sep=0}, "g", from=1-2, to=2-2]
            \arrow["v"'{inner sep = 1ex}, "\shortmid"{marking}, from=2-1, to=2-2]
            \arrow["\alpha"{description}, draw=none, from=0, to=1]
        \end{tikzcd}
    \end{equation*}
    in $\BB$, the action of the lax double functor $F_{2}$ provides a function $F_{2}\alpha \colon F(u) \rightarrow F(v)$ such that $s_{v} \circ F_{2}\alpha = Ff \circ s_{u}$ and $t_{v} \circ F_{2}\alpha = Fg \circ t_{u}$.
    Furthermore, naturality of $\varphi$ implies that the equation $F_{2}\alpha \circ \varphi_{u} = \varphi_{v} \circ Ff$ holds. 
    Thus, for each cell $\alpha$, as above, we have the following cell in $\SMULT$ (where we write $F_{2}\alpha$ as $F\alpha$ for simplicity). 
    \begin{equation*}
        \begin{tikzcd}
            FA & {F(u)} & FB \\
            FC & {F(v)} & FD
            \arrow["{\varphi_{u}}", shift left=2, tail, from=1-1, to=1-2]
            \arrow["Ff"', from=1-1, to=2-1]
            \arrow["{s_{u}}", shift left=2, two heads, from=1-2, to=1-1]
            \arrow["{t_{u}}", from=1-2, to=1-3]
            \arrow["{F\alpha}", from=1-2, to=2-2]
            \arrow["Fg", from=1-3, to=2-3]
            \arrow["{\varphi_{u}}", shift left=2, tail, from=2-1, to=2-2]
            \arrow["{s_{v}}", shift left=2, two heads, from=2-2, to=2-1]
            \arrow["{t_{v}}"', from=2-2, to=2-3]
        \end{tikzcd}
    \end{equation*}
    This completes the specification of the (nominal) lax double functor $F \colon \BB \to \SMULT$ on objects, tight morphisms, loose morphisms, and cells. 
    It remains to define the unit and composition comparison cells which provide the lax structure. 

    We first note that since $\SQ(\Set)$ is a thin double category, $F_{1} \colon \BB \to \SQ(\Set)$ is necessarily a strict double functor. 
    For each object $A$ and each composable pair of loose morphisms $u \colon A \lto B$ and $v \colon B \lto C$ in $\BB$, the unit $\eta$ and composition $\mu$ comparison cells for the lax double functor $F_{2}$ are given by the following diagrams. 
    \begin{equation*}
        \begin{tikzcd}
            FA & FA & FA \\
            FA & {F(\id_{A})} & FA
            \arrow[Rightarrow, no head, from=1-1, to=2-1]
            \arrow["{1_{FA}}"', from=1-2, to=1-1]
            \arrow["{1_{FA}}", from=1-2, to=1-3]
            \arrow["{\eta_{A}}", from=1-2, to=2-2]
            \arrow[Rightarrow, no head, from=1-3, to=2-3]
            \arrow["{s_{A}}", from=2-2, to=2-1]
            \arrow["{t_{A}}"', from=2-2, to=2-3]
        \end{tikzcd}
        \qquad \qquad 
        \begin{tikzcd}
            FA & {F(u, v)} & FC \\
            FA & {F(v \circ u)} & FC
            \arrow[Rightarrow, no head, from=1-1, to=2-1]
            \arrow["{s_{u} \pi_{u}}"', from=1-2, to=1-1]
            \arrow["{t_{v}\pi_{v}}", from=1-2, to=1-3]
            \arrow["{\mu_{(u, v)}}", from=1-2, to=2-2]
            \arrow[Rightarrow, no head, from=1-3, to=2-3]
            \arrow["{s_{u \circ v}}", from=2-2, to=2-1]
            \arrow["{t_{u \circ v}}"', from=2-2, to=2-3]
        \end{tikzcd}
    \end{equation*}
    We adopt the notation $s_{A} \coloneqq s_{\id_{A}}$ and $t_{A} \coloneqq t_{\id_{A}}$ for ease of readability, and let $F(u, v)$ denote the pullback of $t_{u} \colon F(u) \to FB$ along $s_{v} \colon F(v) \to FB$; the corresponding projections $\pi_{u}$ and $\pi_{v}$, respectively. 
    The coherence of the transformation $\varphi$ with the comparison cells implies that $\eta_{A} = \varphi_{\id_{A}}$ and $\varphi_{v \circ u} = \mu_{(u, v)} \circ \langle 1, \varphi_{v} \circ t_{u} \rangle \circ \varphi_{u}$ where $\langle 1, \varphi_{v} \circ t_{u} \rangle \colon F(u) \to F(u, v)$ is the uniquely induced morphism into the pullback. 

    We define the unit and composition comparison cells for the lax double functor $F \colon \BB \to \SMULT$ as follows. 
    For each object $A$ and composable pair of loose morphisms $u \colon A \lto B$ and $v \colon B \lto C$ in $\BB$, we have the following cells in $\SMULT$. 
    \begin{equation*}
        \begin{tikzcd}
            FA & FA & FA \\
            FA & {F(\id_{A})} & FA
            \arrow["{1_{FA}}", shift left=2, tail, from=1-1, to=1-2]
            \arrow[Rightarrow, no head, from=1-1, to=2-1]
            \arrow["{{1_{FA}}}", shift left=2, two heads, from=1-2, to=1-1]
            \arrow["{{1_{FA}}}", from=1-2, to=1-3]
            \arrow["{{\eta_{A}}}", from=1-2, to=2-2]
            \arrow[Rightarrow, no head, from=1-3, to=2-3]
            \arrow["{\varphi_{A}}", shift left=2, tail, from=2-1, to=2-2]
            \arrow["{{s_{A}}}", shift left=2, two heads, from=2-2, to=2-1]
            \arrow["{{t_{A}}}"', from=2-2, to=2-3]
        \end{tikzcd}
        \qquad \qquad 
        \begin{tikzcd}[column sep = large]
            FA & {F(u, v)} & FC \\
            FA & {F(v \circ u)} & FC
            \arrow["{\langle 1, \varphi_{v}t_{u} \rangle \circ \varphi_{u}}", shift left=2, tail, from=1-1, to=1-2]
            \arrow[Rightarrow, no head, from=1-1, to=2-1]
            \arrow["{{s_{u} \pi_{u}}}", shift left=2, two heads, from=1-2, to=1-1]
            \arrow["{{t_{v}\pi_{v}}}", from=1-2, to=1-3]
            \arrow["{{\mu_{u, v}}}", from=1-2, to=2-2]
            \arrow[Rightarrow, no head, from=1-3, to=2-3]
            \arrow["{\varphi_{v \circ u}}", shift left=2, tail, from=2-1, to=2-2]
            \arrow["{{s_{u \circ v}}}", shift left=2, two heads, from=2-2, to=2-1]
            \arrow["{{t_{u \circ v}}}"', from=2-2, to=2-3]
        \end{tikzcd}
    \end{equation*}

    This completes the description of the data required to define the lax double functor $F \colon \BB \to \SMULT$. 
    It is straightforward to show that this lax double functor is well-defined, and that the conditions $U_{1} F = F_{1}$, $U_{2}F = F_{2}$ and $\sigma \cdot F = \varphi$ hold. 
    Proof of uniqueness is similarly straightforward, however we omit the details. 
\end{proof}
    
\begin{theorem}[$2$-dimensional universal property of $\SMULT$]
\label{theorem:2d-universal-property-SMULT}
    Given a double functor $J \colon \BB \to \DD$, lax double functors $F_{1}$, $F_{2}$, $G_{1}$, and $G_{2}$, transformations $\theta_{1}$ and $\theta_{2}$, and globular transformations $\varphi$ and $\psi$ such that 
    \begin{equation*}
        \begin{tikzcd}[column sep = small]
            \BB && \DD \\
            & {\SQ(\Set)} && \SPAN
            \arrow["J", from=1-1, to=1-3]
            \arrow[""{name=0, anchor=center, inner sep=0}, "{F_{1}}"', from=1-1, to=2-2]
            \arrow[""{name=0p, anchor=center, inner sep=0}, phantom, from=1-1, to=2-2, start anchor=center, end anchor=center]
            \arrow[""{name=1, anchor=center, inner sep=0}, "{G_{1}}"{description}, from=1-3, to=2-2]
            \arrow[""{name=1p, anchor=center, inner sep=0}, phantom, from=1-3, to=2-2, start anchor=center, end anchor=center]
            \arrow[""{name=1p, anchor=center, inner sep=0}, phantom, from=1-3, to=2-2, start anchor=center, end anchor=center]
            \arrow[""{name=2, anchor=center, inner sep=0}, "{G_{2}}", from=1-3, to=2-4]
            \arrow[""{name=2p, anchor=center, inner sep=0}, phantom, from=1-3, to=2-4, start anchor=center, end anchor=center]
            \arrow["{K_{\ast}}"', from=2-2, to=2-4]
            \arrow["{\theta_{1}}"', shift left=2, shorten <=8pt, shorten >=8pt, Rightarrow, from=0p, to=1p]
            \arrow["\psi", shift right=2, shorten <=8pt, shorten >=8pt, Rightarrow, from=1p, to=2p]
        \end{tikzcd}
        \quad = \quad
        \begin{tikzcd}[column sep = small]
            & \BB && \DD \\
            {\SQ(\Set)} && \SPAN
            \arrow["J", from=1-2, to=1-4]
            \arrow[""{name=0, anchor=center, inner sep=0}, "{F_{1}}"', from=1-2, to=2-1]
            \arrow[""{name=0p, anchor=center, inner sep=0}, phantom, from=1-2, to=2-1, start anchor=center, end anchor=center]
            \arrow[""{name=1, anchor=center, inner sep=0}, "{F_{2}}"{description}, from=1-2, to=2-3]
            \arrow[""{name=1p, anchor=center, inner sep=0}, phantom, from=1-2, to=2-3, start anchor=center, end anchor=center]
            \arrow[""{name=1p, anchor=center, inner sep=0}, phantom, from=1-2, to=2-3, start anchor=center, end anchor=center]
            \arrow[""{name=2, anchor=center, inner sep=0}, "{G_{2}}", from=1-4, to=2-3]
            \arrow[""{name=2p, anchor=center, inner sep=0}, phantom, from=1-4, to=2-3, start anchor=center, end anchor=center]
            \arrow["{K_{\ast}}"', from=2-1, to=2-3]
            \arrow["\varphi", shift right=2, shorten <=8pt, shorten >=8pt, Rightarrow, from=0p, to=1p]
            \arrow["{\theta_{2}}"', shift left=2, shorten <=7pt, shorten >=7pt, Rightarrow, from=1p, to=2p]
        \end{tikzcd}
    \end{equation*}
    there exists a unique transformation between lax double functors
    \begin{equation*}
        \begin{tikzcd}[column sep=small]
            \BB && \DD \\
            & \SMULT
            \arrow["J", from=1-1, to=1-3]
            \arrow[""{name=0, anchor=center, inner sep=0}, "F"', from=1-1, to=2-2]
            \arrow[""{name=0p, anchor=center, inner sep=0}, phantom, from=1-1, to=2-2, start anchor=center, end anchor=center]
            \arrow[""{name=1, anchor=center, inner sep=0}, "G", from=1-3, to=2-2]
            \arrow[""{name=1p, anchor=center, inner sep=0}, phantom, from=1-3, to=2-2, start anchor=center, end anchor=center]
            \arrow["\theta"', shift left=2, shorten <=8pt, shorten >=8pt, Rightarrow, from=0p, to=1p]
        \end{tikzcd}
    \end{equation*}
    such that $U_{1} \cdot \theta = \theta_{1}$, $U_{2} \cdot \theta = \theta_{2}$, $\sigma \cdot F = \varphi$ and $\sigma \cdot G = \psi$, where $\sigma \colon K_{\ast}U_{1} \Rightarrow U_{2}$ is defined in Lemma~\ref{lemma:globular-SMULT-SPAN-SQSET}.
\end{theorem}
\begin{proof}
    By Theorem~\ref{theorem:1d-universal-property-SMULT}, there exists unique lax double functors $F \colon \BB \to \SMULT$ and $G \colon \DD \to \SMULT$ such that $\sigma \cdot F = \varphi$ and $\sigma \cdot G = \psi$.

    For each loose morphism $u \colon A \lto B$ in $\BB$, the transformation $\theta_{2}$ provides a function $\theta_{2}u \colon F(u) \to GJ(u)$ such that $s_{Ju} \circ \theta_{2}u = \theta_{A} \circ s_{u}$ and $t_{Ju} \circ \theta_{2}u = \theta_{B} \circ t_{u}$. 
    We abuse notation by using $s$ and $t$ for both $F_{2}u$ and $G_{2}u$. 
    By the main assumption, we also have that $\psi_{Ju} \circ \theta_{A} = \theta_{2}u \circ \varphi_{u}$. 
    Therefore, we may define a transformation $\theta \colon F \Rightarrow GJ$ whose component at a loose morphism $u \colon A \lto B$ in $\B$ is given by the following cell in $\SMULT$ (where we write $\theta u$ instead of $\theta_{2}u$ for simplicity). 
    \begin{equation*}
        \begin{tikzcd}
            FA & {F(u)} & FB \\
            GJA & {GJ(u)} & GJB
            \arrow["{\varphi_{u}}", shift left=2, tail, from=1-1, to=1-2]
            \arrow["{\theta_{A}}"', from=1-1, to=2-1]
            \arrow["{s_{u}}", shift left=2, two heads, from=1-2, to=1-1]
            \arrow["{t_{u}}", from=1-2, to=1-3]
            \arrow["{\theta u}"', from=1-2, to=2-2]
            \arrow["{\theta_{B}}", from=1-3, to=2-3]
            \arrow["{\psi_{Ju}}", shift left=2, tail, from=2-1, to=2-2]
            \arrow["{s_{Ju}}", shift left=2, two heads, from=2-2, to=2-1]
            \arrow["{t_{Ju}}"', from=2-2, to=2-3]
        \end{tikzcd}
    \end{equation*}

    Naturality and coherence with the unit and composition comparison cells for $\theta$ is inherited from the transformations $\theta_{1}$ and $\theta_{2}$. 
    Similarly, it is straightforward to show that whiskering $\theta$ with $U_{1} \colon \SMULT \to \SQ(\Set)$ and $U_{2} \colon \SMULT \to \SPAN$ yields $\theta_{1}$ and $\theta_{2}$, respectively. 
    Proof of uniqueness is also easy, and we omit the details. 
\end{proof}

Theorem~\ref{theorem:1d-universal-property-SMULT} and Theorem~\ref{theorem:2d-universal-property-SMULT} completely characterise the double category of split multivalued functions up to isomorphism. 
The universal properties presented are almost identical to those of a \emph{comma double category} \cite{GrandisPare1999}, except that we restrict to a cone which involves a globular transformation. 
An analogous construction of this certain kind of limit is possible by replacing $K_{\ast} \colon \SQ(\Set) \to \SPAN$ with any (colax) double functor, however this level of generality is outside the scope of this paper.

\subsection{Indexed split multivalued functions}

We now use the universal properties of the double category $\SMULT$ to establish an isomorphism of categories. 
We then restrict this along the inclusion $\LO \colon \Cat \to \Dbl$ to obtain the isomorphism required for our main result. 

\begin{definition}
\label{definition:globular-cone-Kast}
    Let $\GlobCone(\Dbl, K_{\ast})$ denote the category whose objects are quadruples $(\BB, F_{1}, F_{2}, \varphi)$ consisting of a double category $\BB$, lax double functors $F_{1} \colon \BB \to \SQ(\Set)$ and $F_{2} \colon \BB \to \SPAN$, and a globular transformation $\varphi \colon K_{\ast} F_{1} \Rightarrow F_{2}$, and whose morphisms $(J, \theta_{1}, \theta_{2}) \colon (\BB, F_{1}, F_{2}, \varphi) \to (\DD, G_{1}, G_{2}, \psi)$ consist of a double functor $J \colon \BB \to \DD$ and tight transformations $\theta_{1} \colon F_{1} \Rightarrow G_{1}J$ and $\theta_{2} \colon F_{2} \Rightarrow G_{2}J$ such that the following equation holds. 
    \begin{equation*}
        \begin{tikzcd}[column sep = small]
            \BB && \DD \\
            & {\SQ(\Set)} && \SPAN
            \arrow["J", from=1-1, to=1-3]
            \arrow[""{name=0, anchor=center, inner sep=0}, "{F_{1}}"', from=1-1, to=2-2]
            \arrow[""{name=0p, anchor=center, inner sep=0}, phantom, from=1-1, to=2-2, start anchor=center, end anchor=center]
            \arrow[""{name=1, anchor=center, inner sep=0}, "{G_{1}}"{description}, from=1-3, to=2-2]
            \arrow[""{name=1p, anchor=center, inner sep=0}, phantom, from=1-3, to=2-2, start anchor=center, end anchor=center]
            \arrow[""{name=1p, anchor=center, inner sep=0}, phantom, from=1-3, to=2-2, start anchor=center, end anchor=center]
            \arrow[""{name=2, anchor=center, inner sep=0}, "{G_{2}}", from=1-3, to=2-4]
            \arrow[""{name=2p, anchor=center, inner sep=0}, phantom, from=1-3, to=2-4, start anchor=center, end anchor=center]
            \arrow["{K_{\ast}}"', from=2-2, to=2-4]
            \arrow["{\theta_{1}}"', shift left=2, shorten <=8pt, shorten >=8pt, Rightarrow, from=0p, to=1p]
            \arrow["\psi", shift right=2, shorten <=8pt, shorten >=8pt, Rightarrow, from=1p, to=2p]
        \end{tikzcd}
        \quad = \quad
        \begin{tikzcd}[column sep = small]
            & \BB && \DD \\
            {\SQ(\Set)} && \SPAN
            \arrow["J", from=1-2, to=1-4]
            \arrow[""{name=0, anchor=center, inner sep=0}, "{F_{1}}"', from=1-2, to=2-1]
            \arrow[""{name=0p, anchor=center, inner sep=0}, phantom, from=1-2, to=2-1, start anchor=center, end anchor=center]
            \arrow[""{name=1, anchor=center, inner sep=0}, "{F_{2}}"{description}, from=1-2, to=2-3]
            \arrow[""{name=1p, anchor=center, inner sep=0}, phantom, from=1-2, to=2-3, start anchor=center, end anchor=center]
            \arrow[""{name=1p, anchor=center, inner sep=0}, phantom, from=1-2, to=2-3, start anchor=center, end anchor=center]
            \arrow[""{name=2, anchor=center, inner sep=0}, "{G_{2}}", from=1-4, to=2-3]
            \arrow[""{name=2p, anchor=center, inner sep=0}, phantom, from=1-4, to=2-3, start anchor=center, end anchor=center]
            \arrow["{K_{\ast}}"', from=2-1, to=2-3]
            \arrow["\varphi", shift right=2, shorten <=8pt, shorten >=8pt, Rightarrow, from=0p, to=1p]
            \arrow["{\theta_{2}}"', shift left=2, shorten <=7pt, shorten >=7pt, Rightarrow, from=1p, to=2p]
        \end{tikzcd}
    \end{equation*}
\end{definition}

The notation for $\GlobCone(\Dbl, K_{\ast})$ is chosen to remind us that the objects are globular cones over $K_{\ast}$, and that there is a fibration $\GlobCone(\Dbl, K_{\ast}) \to \Dbl$ to the category of double categories and double functors. 

\begin{definition}
\label{definition:Idx-Dbl-SMULT}
    Let $\Idx(\Dbl, \SMULT)$ denote the category whose objects are pairs $(\BB, F)$ consisting of a double category $\B$ and a lax double functor $F \colon \BB \to \SMULT$, and whose morphisms $(J, \theta) \colon (\BB, F) \to (\DD, G)$ consist of a double functor $J \colon \BB \to \DD$ and a tight transformation $\theta \colon F \Rightarrow GJ$. 
\end{definition}

The notation $\Idx(\Dbl, \SMULT)$ is chosen to remind us that the objects are split multivalued functions ``indexed'' by a double category. 
There is also a canonical fibration $\Idx(\Dbl, \SMULT) \to \Dbl$. 

\begin{corollary}
\label{corollary:globular-cones-lax-double-functors}
    There is an isomorphism $\GlobCone(\Dbl, K_{\ast}) \cong \Idx(\Dbl, \SMULT)$.
\end{corollary}
\begin{proof}
    This follows immediately from the universal properties of $\SMULT$ exhibited in Theorem~\ref{theorem:1d-universal-property-SMULT} and Theorem~\ref{theorem:2d-universal-property-SMULT}.
\end{proof}

Recall that there is a fully faithful functor $\LO \colon \Cat \to \Dbl$ which constructs a strict double category $\LO(\B)$ whose objects and loose morphisms are the objects and morphisms of the category $\B$, and whose tight morphisms and cells are identities. 
We formally define the following useful terminology introduced in Subsection~\ref{subsec:overview}.

\begin{definition}
\label{definition:indexed-split-mutlivalued-function}
    A lax double functor $\LO(\B) \to \SMULT$ is called an \emph{indexed split multivalued function}. 
\end{definition}

\begin{remark}
\label{remark:restriction-to-Cat}
We may take the pullback of the fibration $\Idx(\Dbl, \SMULT) \to \Dbl$ along the functor $\LO \colon \Cat \to \Dbl$ to obtain the category of indexed split multivalued functions which we denote $\Idx(\Cat, \SMULT)$. 
Explicitly, the objects are pairs $(\B, F)$ of a category $\B$ and a lax double functor $F \colon \LO(\B) \to \SMULT$, while the morphisms $(k, \theta) \colon (\B, F) \to (\D, G)$ consist of a functor $k \colon \B \to \D$ and a tight transformation $\theta \colon F \Rightarrow G \circ \LO(k)$. 
The category $\GlobCone(\Cat, K_{\ast})$ may be defined by an analogous pullback. 
\end{remark}

The following corollary is an important component for proving our main theorem. 

\begin{corollary}
\label{corollary:Idx-GlobCone}
    There is an isomorphism $\GlobCone(\Cat, K_{\ast}) \cong \Idx(\Cat, \SMULT)$.
\end{corollary}

%% file: category-of-elements.tex
\section{The category of elements construction for delta lenses}
\label{sec:main-theorem}

In this section, we first establish an equivalence $\DiaLens \simeq \GlobCone(\Cat, K_{\ast})$ between the categories of diagrammatic delta lenses and globular cones over $K_{\ast}$. 
We then prove our main result: an equivalence $\Lens \simeq \Idx(\Cat, \SMULT)$ between the categories of delta lenses and indexed split multivalued functions.

\subsection{A double-categorical approach to the category of elements}
\label{sec:double-cat-of-elements}
We begin by extending the category $\Idx(\Cat, \SMULT)$ to one whose objects are lax double functors valued in an arbitrary double category. 
We then use this setting to provide a double-categorical perspective on the indexed presentations of discrete opfibrations and functors. 

\begin{definition}
\label{definition:indexed-categories}
    Let $\Idx(\Cat, \XX)$ be the category whose objects are pairs $(\B, F)$ of a (small) category $\B$ and a lax double functor $F \colon \LO(\B) \to \XX$, and whose morphisms $(k, \theta) \colon (\B, F) \to (\D, G)$ consist of a functor $k \colon \B \to \D$ and a tight transformation of double functors as depicted below. 
    \begin{equation*}
        \begin{tikzcd}[column sep=small]
            {\LO(\B)} && {\LO(\D)} \\
            & \XX
            \arrow["{\LO(k)}", from=1-1, to=1-3]
            \arrow[""{name=0, anchor=center, inner sep=0}, "F"', from=1-1, to=2-2]
            \arrow[""{name=1, anchor=center, inner sep=0}, "G", from=1-3, to=2-2]
            \arrow["\theta", shift left=2, shorten <=8pt, shorten >=8pt, Rightarrow, from=0, to=1]
        \end{tikzcd}
    \end{equation*}
\end{definition}

Using this definition, we can restate the classical category of elements construction for $\Set$-valued functors in terms of double categories. 
Recall that $\DOpf$ is the full subcategory of $\Cat^{\two}$ whose objects are discrete opfibrations. 

\begin{proposition}
\label{proposition:category-of-elements-SQSET}
    There is an equivalence of categories $\DOpf \simeq \Idx(\Cat, \SQ(\Set))$.
\end{proposition}
\begin{proof}
    Given a category $\B$, a lax functor $\LO(\B) \to \SQ(\Set)$ is necessarily strict, and moreover equivalent to a functor $\B \to \Set$. 
    A morphism in $\Idx(\Cat, \SQ(\Set))$ is equivalent to a functor $\B \to \D$ and a natural transformation as shown below. 
    \begin{equation}
    \label{equation:category-of-elements-SQSET}
    \begin{tikzcd}[column sep=small]
        \B && \D \\
        & \Set
        \arrow[from=1-1, to=1-3]
        \arrow[""{name=0, anchor=center, inner sep=0}, from=1-1, to=2-2]
        \arrow[""{name=1, anchor=center, inner sep=0}, from=1-3, to=2-2]
        \arrow[shift left=2, shorten <=6pt, shorten >=6pt, Rightarrow, from=0, to=1]
    \end{tikzcd}
    \end{equation}
    By the classical category of elements construction, functors $\B \to \Set$ are the same as discrete opfibrations into $\B$, and natural transformations \eqref{equation:category-of-elements-SQSET} are the same as morphisms in $\DOpf$. 
\end{proof}

A generalisation of the classical category of elements construction provides an equivalence between lax double functors $\LO(\B) \to \SPAN$ and ordinary functors into~$\B$, as shown by Paré \cite[Example~3.13]{Pare2011} and Pavlović-Abramsky \cite[Proposition~4]{PavlovicAbramsky1997}.
We restate this result without the requirement of a fixed base category. 

\begin{proposition}
\label{proposition:category-of-elements-SPAN}
    There is an equivalence of categories $\Cat^{\two} \simeq \Idx(\Cat, \SPAN)$. 
\end{proposition}

Given a lax double functor $F \colon \XX \to \YY$, there is an induced functor
\[
    F \colon \Idx(\Cat, \XX) \to \Idx(\Cat, \YY)
\]
given by post-composition (also denoted $F$ by abuse of notation).
If we consider the strict double functor $K_{\ast} \colon \SQ(\Set) \to \SPAN$, defined in \eqref{equation:SQSET-to-SPAN}, this induces a fully faithful functor $K_{\ast} \colon \Idx(\Cat, \SQ(\Set)) \to \Idx(\Cat, \SPAN)$ which ``commutes'' with the category of elements construction in Proposition~\ref{proposition:category-of-elements-SQSET} and Proposition~\ref{proposition:category-of-elements-SPAN}.

To be precise, let $\el \colon \Idx(\Cat, \SQ(\Set)) \to \DOpf$ and $\el \colon \Idx(\Cat, \SPAN) \to \Cat^{\two}$ denote the previous equivalences.
Given a morphism $(k, \theta) \colon (\B, F) \to (\D, G)$ in $\Idx(\Cat, \SPAN)$, its image of under functor $\el$ is the following morphism in $\Cat^{\two}$.
\begin{equation*}
    \begin{tikzcd}
        {\el(\B, F)} & {\el(\D, G)} \\
        \B & \D
        \arrow["{\El(k, \theta)}", from=1-1, to=1-2]
        \arrow["{\pi_{(\B, F)}}"', from=1-1, to=2-1]
        \arrow["{\pi_{(\D, G)}}", from=1-2, to=2-2]
        \arrow["k"', from=2-1, to=2-2]
    \end{tikzcd}
\end{equation*}
The following result tells us that discrete opfibrations correspond precisely to those lax double functors $\LO(\B) \to \SMULT$ which factor through $K_{\ast} \colon \SQ(\Set) \to \SPAN$. 

\begin{lemma}
\label{lemma:category-of-elements-commute}
    The following diagram commutes up to natural isomorphism.
    \begin{equation*}
        \begin{tikzcd}
            {\Idx(\Cat, \SQ(\Set))} & {\Idx(\Cat, \SPAN)} \\
            \DOpf & {\Cat^{\two}}
            \arrow[hook, "{K_{\ast}}", from=1-1, to=1-2]
            \arrow["\el"', "\simeq", from=1-1, to=2-1]
            \arrow["\el", "\simeq"', from=1-2, to=2-2]
            \arrow[hook, from=2-1, to=2-2]
        \end{tikzcd}
    \end{equation*}
\end{lemma}

We now address the relationship between identity-on-objects functors and globular transformations (see Definition~\ref{definition:globular-transformation}) under the category of elements construction. 

\begin{lemma}
\label{lemma:globular-to-ioo}
    If $(1_{\B}, \theta) \colon (\B, F) \to (\B, G)$ is a morphism in $\Idx(\Cat, \SPAN)$ such that $\theta$ is a globular transformation, then $\El(1_{\B}, \theta) \colon \El(\B, F) \to \El(\B, G)$ is an identity-on-objects functor. 

    Conversely, if $(h, 1_{\B})$ is a morphism in $\Cat^{\two}$ such that $h$ is identity-on-objects, then its image under the equivalence $\Cat^{\two} \simeq \Idx(\Cat, \SPAN)$ is a globular transformation.
\end{lemma}
\begin{proof}
    The component of globular transformation $\theta$ at an object $b$ in $\B$ is the identity function $F(b) = G(b)$ in $\Set$. 
    By construction, an object in $\el(\B, F)$ is a pair $(b \in \B, x \in F(b))$, and this is sent by the functor $\el(1_{\B}, \theta)$ to the object $(b \in \B, \theta_{b}(x) \in G(b))$ in $\el(\B, G)$. 
    However, since $\theta$ is globular, $\theta_{b}(x) = x$ and the functor $\el(\B, \theta)$ is identity-on-objects. 
    The proof of the converse is also straightforward, and is omitted.
\end{proof}

\begin{theorem}
\label{theorem:DiaLens-GlobCone}
    There is an equivalence $\DiaLens \simeq \GlobCone(\Cat, K_{\ast})$. 
\end{theorem}
\begin{proof}
Recall from Definition~\ref{definition:DiaLens} that a diagrammatic delta lens $(f, p)$ is a composable pair of functors $f \colon \A \to \B$ and $p \colon \X \to \A$ such that $p$ is identity-on-objects and $fp$ is a discrete opfibration. 
By Lemma~\ref{lemma:globular-to-ioo} and Proposition~\ref{proposition:category-of-elements-SPAN}, a diagrammatic delta lens corresponds to a globular transformation $F_{1} \Rightarrow F_{2}$ in $\Idx(\Cat, \SPAN)$. 
By Lemma~\ref{lemma:category-of-elements-commute} and Proposition~\ref{proposition:category-of-elements-SQSET}, the double functor $F_{1} \colon \LO(\B) \to \SMULT$ factors through $K_{\ast} \colon \SQ(\Set) \to \SPAN$. 
Therefore, each diagrammatic delta lens corresponds to a globular transformation $K_{\ast} \circ F_{1}' \Rightarrow F_{2}$ in $\Idx(\Cat, \SPAN)$ which is precisely an object in $\GlobCone(\Cat, K_{\ast})$ by Definition~\ref{definition:globular-cone-Kast} and Remark~\ref{remark:restriction-to-Cat}.

A morphism in $\DiaLens$ corresponds to a morphism in $\Cat^{\two}$ and a morphism in $\DOpf$ which are suitably compatible. 
Similarly, a morphism in $\GlobCone(\Cat, K_{\ast})$ corresponds to a morphism in $\Idx(\Cat, \SPAN)$ and a morphism in $\Idx(\Cat, \SQ(\Set))$ which are suitably compatible. 
By Proposition~\ref{proposition:category-of-elements-SQSET} and Proposition~\ref{proposition:category-of-elements-SPAN}, we have equivalences $\DOpf \simeq \Idx(\Cat, \SQ(\Set))$ and $\Cat^{\two} \simeq \Idx(\Cat, \SPAN)$, respectively, and it straightforward to show that these extend to the desired equivalence $\DiaLens \simeq \GlobCone(\Cat, K_{\ast})$ by diagram-chasing and Lemma~\ref{lemma:category-of-elements-commute}. 
\end{proof}

\subsection{Main theorem: the Grothendieck construction for delta lenses}

In Definition~\ref{definition:Lens-category}, we defined the category $\Lens$ whose objects are delta lenses (see Definition~\ref{definition:delta-lens}). 
In Remark~\ref{remark:restriction-to-Cat}, we defined the category $\Idx(\Cat, \SMULT)$ whose objects are lax double functors $\LO(\B) \to \SMULT$, also called indexed split multivalued functions (see Definition~\ref{definition:indexed-split-mutlivalued-function}). 
We now show that these categories are equivalent. 

\begin{theorem}
\label{theorem:main}
    There is an equivalence of categories 
    \[
        \el \colon \Idx(\Cat, \SMULT) \simeq \Lens 
    \]
    between the categories of indexed split multivalued functions and delta lenses. 
\end{theorem}
\begin{proof}
We have the following sequence of equivalences
\[
    \Idx(\Cat, \SMULT) \cong \GlobCone(\Cat, K_{\ast}) \simeq \DiaLens \simeq \Lens
\]
by Theorem~\ref{theorem:Lens-DiaLens}, Theorem~\ref{theorem:DiaLens-GlobCone}, and Corollary~\ref{corollary:Idx-GlobCone}. 
\end{proof}

The functor $\el \colon \Idx(\Cat, \SMULT) \to \Lens$ is called the \emph{Grothendieck construction} for delta lenses, or alternatively, the \emph{category of elements} of an indexed split multivalued function. 
The action of this construction was sketched in Subsection~\ref{sec:proof-sketch}.

%% file: examples.tex
\section{Examples and applications}
\label{sec:examples}

\subsection{Adjunctions of double categories}
\label{sec:Adjunctions}

In Proposition~\ref{proposition:DOpf-Lens}, we showed that there is a coreflective adjunction
\begin{equation*}
    \begin{tikzcd}
        \DOpf & \Lens
        \arrow[""{name=0, anchor=center, inner sep=0}, "\Upsilon"', shift right=2, hook, from=1-1, to=1-2]
        \arrow[""{name=1, anchor=center, inner sep=0}, "\Lambda"', shift right=2, from=1-2, to=1-1]
        \arrow["\dashv"{anchor=center, rotate=90}, draw=none, from=0, to=1]
    \end{tikzcd}
\end{equation*}
between the category of discrete opfibrations and the category of delta lenses. 
We may demonstrate that this is induced by an adjunction of \emph{double categories}. 

Recall that the double category $\SMULT$ admits a $1$-dimensional and $2$-dimensional universal property as the universal globular cone over $K_{\ast}$. 
\begin{equation*}
    \begin{tikzcd}[column sep = small]
        & \SMULT \\
        {\SQ(\Set)} && \SPAN
        \arrow[""{name=0, anchor=center, inner sep=0}, "{U_{1}}"', from=1-2, to=2-1]
        \arrow[""{name=1, anchor=center, inner sep=0}, "{U_{2}}", from=1-2, to=2-3]
        \arrow["{K_{\ast}}"', from=2-1, to=2-3]
        \arrow["\sigma", shift right=2, shorten <=10pt, shorten >=10pt, Rightarrow, from=0, to=1]
    \end{tikzcd}
\end{equation*}
Given the identity (globular) transformation on $K_{\ast}$, there is a unique (strict) double functor $\langle 1, K_{\ast} \rangle \colon \SQ(\Set) \to \SMULT$ such that $U_{1} \langle 1, K_{\ast} \rangle = 1_{\SQ(\Set)}$, $U_{2} \langle 1, K_{\ast} \rangle = K_{\ast}$ and $\sigma \cdot \langle 1, K_{\ast} \rangle = \id_{K_{\ast}}$ by Theorem~\ref{theorem:1d-universal-property-SMULT}. 
Given the morphism 
\begin{equation*}
    \begin{tikzcd}[column sep = tiny]
        \SMULT && \SMULT \\
        & {\SQ(\Set)} && \SPAN
        \arrow["{1_{\SMULT}}", from=1-1, to=1-3]
        \arrow[""{name=0, anchor=center, inner sep=0}, "{U_{1}}"', from=1-1, to=2-2]
        \arrow[""{name=0p, anchor=center, inner sep=0}, phantom, from=1-1, to=2-2, start anchor=center, end anchor=center]
        \arrow[""{name=1, anchor=center, inner sep=0}, "{U_{1}}"{description}, from=1-3, to=2-2]
        \arrow[""{name=1p, anchor=center, inner sep=0}, phantom, from=1-3, to=2-2, start anchor=center, end anchor=center]
        \arrow[""{name=1p, anchor=center, inner sep=0}, phantom, from=1-3, to=2-2, start anchor=center, end anchor=center]
        \arrow[""{name=2, anchor=center, inner sep=0}, "{U_{2}}", from=1-3, to=2-4]
        \arrow[""{name=2p, anchor=center, inner sep=0}, phantom, from=1-3, to=2-4, start anchor=center, end anchor=center]
        \arrow["{{K_{\ast}}}"', from=2-2, to=2-4]
        \arrow["\id"', shift left=2, shorten <=10pt, shorten >=10pt, Rightarrow, from=0p, to=1p]
        \arrow[draw=none, from=1, to=2-2]
        \arrow["\sigma", shift right=2, shorten <=12pt, shorten >=12pt, Rightarrow, from=1p, to=2p]
    \end{tikzcd}
    \quad = \quad
    \begin{tikzcd}[column sep = tiny]
        & \SMULT && \SMULT \\
        {\SQ(\Set)} && \SPAN
        \arrow["{1_{\SMULT}}", from=1-2, to=1-4]
        \arrow[""{name=0, anchor=center, inner sep=0}, "{U_{1}}"', from=1-2, to=2-1]
        \arrow[""{name=0p, anchor=center, inner sep=0}, phantom, from=1-2, to=2-1, start anchor=center, end anchor=center]
        \arrow[""{name=1, anchor=center, inner sep=0}, "{K_{\ast}U_{1}}"{description}, from=1-2, to=2-3]
        \arrow[""{name=1p, anchor=center, inner sep=0}, phantom, from=1-2, to=2-3, start anchor=center, end anchor=center]
        \arrow[""{name=1p, anchor=center, inner sep=0}, phantom, from=1-2, to=2-3, start anchor=center, end anchor=center]
        \arrow[""{name=2, anchor=center, inner sep=0}, "{U_{2}}", from=1-4, to=2-3]
        \arrow[""{name=2p, anchor=center, inner sep=0}, phantom, from=1-4, to=2-3, start anchor=center, end anchor=center]
        \arrow["{{K_{\ast}}}"', from=2-1, to=2-3]
        \arrow["\id", shift right=2, shorten <=9pt, shorten >=9pt, Rightarrow, from=0p, to=1p]
        \arrow["\sigma"', shift left=2, shorten <=9pt, shorten >=9pt, Rightarrow, from=1p, to=2p]
    \end{tikzcd}
\end{equation*}
in $\GlobCone(\Cat, K_{\ast})$, there exists a unique tight transformation
\begin{equation}
\label{equation:counit-SQSET-SMULT}
    \begin{tikzcd}[column sep = tiny]
        \SMULT && \SMULT \\
        & \SMULT
        \arrow["{1_{\SMULT}}", from=1-1, to=1-3]
        \arrow[""{name=0, anchor=center, inner sep=0}, "{\langle 1, K_{\ast} \rangle U_{1}}"', from=1-1, to=2-2]
        \arrow[""{name=1, anchor=center, inner sep=0}, "{1_{\SMULT}}", from=1-3, to=2-2]
        \arrow["\varepsilon", shorten <=10pt, shorten >=10pt, Rightarrow, from=0, to=1]
    \end{tikzcd}
\end{equation}
such that $U_{1} \cdot \varepsilon = 1_{U_{1}}$ and $U_{2} \cdot \varepsilon = \sigma$ by Theorem~\ref{theorem:2d-universal-property-SMULT}. 
Applying the universal property again, we can show that $\varepsilon \cdot \langle 1, K_{\ast} \rangle = 1_{\langle 1, K_{\ast} \rangle}$, yielding the following result.

\begin{proposition}
\label{proposition:SQSET-SMULT-adjunction}
    There is a coreflective adjunction
    \begin{equation*}
        \begin{tikzcd}[column sep = large]
            {\SQ(\Set)} & \SMULT
            \arrow[""{name=0, anchor=center, inner sep=0}, "{\langle 1, K_{\ast} \rangle}"', shift right=2, hook, from=1-1, to=1-2]
            \arrow[""{name=1, anchor=center, inner sep=0}, "{U_{1}}"', shift right=2, from=1-2, to=1-1]
            \arrow["\dashv"{anchor=center, rotate=90}, draw=none, from=0, to=1]
        \end{tikzcd}
    \end{equation*}
    between the double category of commutative squares in $\Set$ and the double category of split multivalued functions, whose counit is \eqref{equation:counit-SQSET-SMULT}.
\end{proposition}

The category $\Idx(\Cat, \XX)$ introduced in Definition~\ref{definition:indexed-categories} extends to a $2$-functor $\Idx(\Cat, -) \colon \Dbl_{\lax} \to \CAT$ from the $2$-category of double categories, lax functors, and tight transformations to the $2$-category of locally small categories, functors, and natural transformations
Since $\DOpf \simeq \Idx(\Cat, \SQ(\Set))$ by Proposition~\ref{proposition:category-of-elements-SQSET} and $\Lens \simeq \Idx(\Cat, \SMULT)$ by Theorem~\ref{theorem:main}, we recover the coreflective adjunction of Proposition~\ref{proposition:DOpf-Lens} as an immediate corollary of Proposition~\ref{proposition:SQSET-SMULT-adjunction}.

The counit $\varepsilon$ defined in \eqref{equation:counit-SQSET-SMULT} has a component at a split multivalued function $(s, X, t, \sigma) \colon A \lto B$ given by the following cell. 
\begin{equation}
\label{equation:counit-component}
    \begin{tikzcd}
        A & A & B \\
        A & X & B
        \arrow["{1_{A}}", shift left=2, tail, from=1-1, to=1-2]
        \arrow["{1_{A}}", from=1-1, to=2-1]
        \arrow["{1_{A}}", shift left=2, two heads, from=1-2, to=1-1]
        \arrow["{t\sigma}", from=1-2, to=1-3]
        \arrow["\sigma", from=1-2, to=2-2]
        \arrow["{1_{B}}", from=1-3, to=2-3]
        \arrow["\sigma", shift left=2, from=2-1, to=2-2]
        \arrow["s", shift left=2, two heads, from=2-2, to=2-1]
        \arrow["t"', from=2-2, to=2-3]
    \end{tikzcd}
\end{equation}

\begin{proposition}
\label{proposition:reflective-adjunction-SQSET-SMULT}
    There is a reflective adjunction 
    \[
        \begin{tikzcd}
            {\SQ(\Set)} & \SMULT
            \arrow[""{name=0, anchor=center, inner sep=0}, shift right=2, hook, from=1-1, to=1-2]
            \arrow[""{name=1, anchor=center, inner sep=0}, "{{U_{1}}}"', shift right=2, from=1-2, to=1-1]
            \arrow["\dashv"{anchor=center, rotate=-90}, draw=none, from=1, to=0]
        \end{tikzcd}
    \]
    whose right adjoint is the identity on objects and tight morphisms, and whose assignment on loose morphisms sends a function $A \to B$ to the split multivalued function $(\pi_{A}, A \times B, \pi_{B}, \langle 1_{A}, f \rangle) \colon A \lto B$. 
\end{proposition}

The component of the unit of this adjunction at a split multivalued function $(s, X, t, \sigma) \colon A \lto B$ is given by the following cell. 
\begin{equation*}
    \begin{tikzcd}
        {A } & X & B \\
        A & {A \times B} & B
        \arrow["\sigma", shift left=2, tail, from=1-1, to=1-2]
        \arrow["{1_{A}}"', from=1-1, to=2-1]
        \arrow["s", shift left=2, two heads, from=1-2, to=1-1]
        \arrow["t", from=1-2, to=1-3]
        \arrow["{\langle s, t \rangle}", from=1-2, to=2-2]
        \arrow["{1_{B}}", from=1-3, to=2-3]
        \arrow["{\langle 1_{A}, t \sigma \rangle}", shift left=2, tail, from=2-1, to=2-2]
        \arrow["{\pi_{B}}", shift left=2, two heads, from=2-2, to=2-1]
        \arrow["{\pi_{B}}"', from=2-2, to=2-3]
    \end{tikzcd}
\end{equation*}
Applying the $2$-functor $\Idx(\Cat, -)$ to this reflective adjunction of double categories induces an adjunction of categories 
\[
    \begin{tikzcd}
        \DOpf & \Lens
        \arrow[""{name=0, anchor=center, inner sep=0}, shift right=2, hook, from=1-1, to=1-2]
        \arrow[""{name=1, anchor=center, inner sep=0}, "{U_{1}}"', shift right=2, from=1-2, to=1-1]
        \arrow["\dashv"{anchor=center, rotate=-90}, draw=none, from=1, to=0]
    \end{tikzcd}
\]
by Theorem~\ref{theorem:main}, whose unit factorises each delta lens into an \emph{identity-on-objects functor} followed by a \emph{fully faithful delta lens} as discussed in Remark~\ref{remark:DOpf-Lens}.

There is also an adjunction of categories 
\[
    \begin{tikzcd}
        \two & \Set
        \arrow[""{name=0, anchor=center, inner sep=0}, shift right=2, hook, from=1-1, to=1-2]
        \arrow[""{name=1, anchor=center, inner sep=0}, shift right=2, from=1-2, to=1-1]
        \arrow["\dashv"{anchor=center, rotate=-90}, draw=none, from=1, to=0]
    \end{tikzcd}
\]
between the interval category $\two = \{ \bot \to \top \}$ and $\Set$, whose fully faithful right adjoint sends $\bot$ to the empty set and $\top$ to a singleton set, and whose left adjoint sends a set to $\bot$ if it is the empty set and to $\top$ otherwise. 

\begin{corollary}
\label{corollary:SQ2-SQSET-SMULT}
    There is a composable pair of adjunctions as follows. 
    \[
        \begin{tikzcd}
            {\SQ(\two)} & {\SQ(\Set)} & \SMULT
            \arrow[""{name=0, anchor=center, inner sep=0}, shift right=2, hook, from=1-1, to=1-2]
            \arrow[""{name=1, anchor=center, inner sep=0}, shift right=2, from=1-2, to=1-1]
            \arrow[""{name=2, anchor=center, inner sep=0}, shift right=2, hook, from=1-2, to=1-3]
            \arrow[""{name=3, anchor=center, inner sep=0}, "{U_{1}}"', shift right=2, from=1-3, to=1-2]
            \arrow["\dashv"{anchor=center, rotate=-90}, draw=none, from=1, to=0]
            \arrow["\dashv"{anchor=center, rotate=-90}, draw=none, from=3, to=2]
        \end{tikzcd}
    \]
\end{corollary}

The composite adjunction induces a monad on $\SMULT$, which under the $2$-functor $\Idx(\Cat, -)$ induces a monad on $\Lens$ by Theorem~\ref{theorem:main}.
The unit of this adjunction factorises each delta lens into a \emph{surjective-on-objects delta lens} followed by a \emph{fully faithful injective-on-objects delta lens}.  
This recovers the (epi, mono)-factorisation on the category of small categories and delta lenses \cite[Theorem~4.17]{CholletClarkeJohnsonSongaWangZardini2022}.

\subsection{Split opfibrations as indexed split multivalued functions}

In Section~\ref{sec:split-opfibration}, we defined split opfibrations as delta lenses with the property that the chosen lifts are opcartesian, thus form a full subcategory $\SOpf \hookrightarrow \Lens$.
We also recalled the characterisation of split opfibrations as diagrammatic delta lenses with a certain property (Proposition~\ref{proposition:diagrammatic-SOpf}). 
In the same spirit as Section~\ref{sec:double-cat-of-elements}, we can state the Grothendieck construction for split opfibrations in terms of double categories. 

\begin{proposition}
\label{proposition:classical-Grothendieck}
    There is an equivalence of categories $\SOpf \simeq \Idx(\Cat, \SQ(\Cat))$.
\end{proposition}
\begin{proof}
    Given a category $\B$, a lax functor $\LO(\B) \to \SQ(\Cat)$ is necessarily strict, and moreover equivalent to a functor $\B \to \Cat$. 
    By the Grothendieck construction, functors $\B \to \Cat$ correspond to split opfibrations into $\B$. 
    A similar argument can be made for the morphisms, and we omit the details.
\end{proof}

By Theorem~\ref{theorem:main}, it must also be possible to characterise split opfibrations as indexed split multivalued functions with a certain property. 
Given a lax double functor $F \colon \LO(\B) \to \SMULT$, we may construct the following cell in $\SMULT$ for each morphism $u \colon x \to y$ in $\B$, where $\varepsilon_{u}$ is the component of the counit \eqref{equation:counit-component}, and $\mu_{(u, 1_{y})}$ is the composition comparison cell of the lax double functor.

\begin{equation}
\label{equation:split-opfibration-cell}
    \begin{tikzcd}[column sep=large]
        {F(x)} & {F(y)} & {F(y)} \\
        {F(x)} & {F(y)} & {F(y)} \\
        {F(x)} && {F(y)}
        \arrow["{\langle 1, K_{\ast}\rangle U_{1}F(u)}", "\shortmid"{marking}, from=1-1, to=1-2]
        \arrow[""{name=0, anchor=center, inner sep=0}, equals, from=1-1, to=2-1]
        \arrow["{F(1_{y})}", "\shortmid"{marking}, from=1-2, to=1-3]
        \arrow[""{name=1, anchor=center, inner sep=0}, equals, from=1-2, to=2-2]
        \arrow[""{name=2, anchor=center, inner sep=0}, equals, from=1-3, to=2-3]
        \arrow["{F(u)}"', "\shortmid"{marking}, from=2-1, to=2-2]
        \arrow[""{name=3, anchor=center, inner sep=0}, equals, from=2-1, to=3-1]
        \arrow["{F(1_{y})}"', "\shortmid"{marking}, from=2-2, to=2-3]
        \arrow[""{name=4, anchor=center, inner sep=0}, equals, from=2-3, to=3-3]
        \arrow["{F(u)}"', "\shortmid"{marking}, from=3-1, to=3-3]
        \arrow["{\varepsilon_{u}}"{description}, draw=none, from=0, to=1]
        \arrow["1"{description}, draw=none, from=1, to=2]
        \arrow["{\mu_{(u, 1_{y})}}"{description}, draw=none, from=3, to=4]
    \end{tikzcd}
\end{equation}

\begin{proposition}
\label{proposition:split-opfibration-as-ismf}
    A lax double functor $F \colon \LO(\B) \to \SMULT$ corresponds to a split opfibration under the equivalence $\Lens \simeq \Idx(\Cat, \SMULT)$ if and only if the cell \eqref{equation:split-opfibration-cell} is invertible (in the tight direction). 
\end{proposition}
\begin{proof}
    Consider the composite of the loose morphisms $\langle 1, K_{\ast}\rangle U_{1}F(u) \colon F(x) \lto F(y)$ and $F(1_{y}) \colon F(y) \lto F(y)$ in $\SMULT$, where $F(u)$ is the split multivalued function $(s_{u}, F(u), t_{u}, \varphi_{u})$, and we use notation $s_{y} \coloneqq s_{1_{y}}$, $t_{y} \coloneqq t_{1_{y}}$ and $\varphi_{y} \coloneqq \varphi_{1_{y}}$ for clarity. 
    \begin{equation*}
        \begin{tikzcd}
            && {F(x, 1_{y})} \\
            & {F(x)} && {F(1_{y})} \\
            {F(x)} && {F(y)} && {F(y)}
            \arrow["{\pi_{x}}", shift left=2, two heads, from=1-3, to=2-2]
            \arrow["{\pi_{y}}", from=1-3, to=2-4]
            \arrow["\lrcorner"{anchor=center, pos=0.125, rotate=-45}, draw=none, from=1-3, to=3-3]
            \arrow["{\langle 1_{F(x)}, \varphi_{y}t_{u}\varphi_{u} \rangle}", shift left=2, tail, from=2-2, to=1-3]
            \arrow["{1_{F(x)}}", shift left=2, two heads, from=2-2, to=3-1]
            \arrow["{t_{u}\varphi_{u}}", from=2-2, to=3-3]
            \arrow["{s_{y}}", shift left=2, two heads, from=2-4, to=3-3]
            \arrow["{t_{y}}", from=2-4, to=3-5]
            \arrow["{1_{F(x)}}", shift left=2, tail, from=3-1, to=2-2]
            \arrow["{\varphi_{y}}", shift left=2, tail, from=3-3, to=2-4]
        \end{tikzcd}
    \end{equation*}
    Next we may consider the composite of the loose morphisms $F(u) \colon F(x) \lto F(y)$ and $F(1_{y}) \colon F(y) \lto F(y)$ in $\SMULT$. 
    \begin{equation*}
        \begin{tikzcd}
            && {F(u, 1_{y})} \\
            & {F(u)} && {F(1_{y})} \\
            {F(x)} && {F(y)} && {F(y)}
            \arrow["{\pi_{u}}", shift left=2, two heads, from=1-3, to=2-2]
            \arrow["{\pi_{y}'}", from=1-3, to=2-4]
            \arrow["\lrcorner"{anchor=center, pos=0.125, rotate=-45}, draw=none, from=1-3, to=3-3]
            \arrow["{\langle 1_{F(u)}, \varphi_{y}t_{u} \rangle}", shift left=2, tail, from=2-2, to=1-3]
            \arrow["{s_{u}}", shift left=2, two heads, from=2-2, to=3-1]
            \arrow["{t_{u}}", from=2-2, to=3-3]
            \arrow["{s_{y}}", shift left=2, two heads, from=2-4, to=3-3]
            \arrow["{t_{y}}", from=2-4, to=3-5]
            \arrow["{\varphi_{u}}", shift left=2, tail, from=3-1, to=2-2]
            \arrow["{\varphi_{y}}", shift left=2, tail, from=3-3, to=2-4]
        \end{tikzcd}
    \end{equation*}
    There is a function $\varphi_{u} \times 1 \colon F(x, 1_{y}) \to F(u, 1_{y})$ such that $\pi_{u}(\varphi_{u} \times 1) = \varphi_{u} \pi_{x}$ and $\pi_{y}'(\varphi_{u} \times 1) = \pi_{y}$ by the universal property of the pullback, and a function $\mu_{(u, 1_{y})} \colon F(u, 1_{y}) \to F(u)$ from the composition comparison of the lax double functor. 
    
    The cell \eqref{equation:split-opfibration-cell} is invertible if and only if the function 
    \[
        \mu_{(u, 1_{y})}(\varphi_{u} \times 1) \colon F(x, 1_{y}) \to F(u)
    \]
    is a bijection, which holds if and only if there exists a function $\chi_{u} \colon F(u) \to F(1_{y})$ rendering the following three diagrams in $\Set$ commutative. 
    \begin{equation*}
        \begin{tikzcd}
            {F(u)} & {F(1_{y})} \\
            {F(x)} & {F(y)}
            \arrow["{\chi_{u}}", from=1-1, to=1-2]
            \arrow["{s_{u}}"', from=1-1, to=2-1]
            \arrow["{s_{y}}", from=1-2, to=2-2]
            \arrow["{t_{u}\varphi_{u}}"', from=2-1, to=2-2]
        \end{tikzcd}
        \qquad
        \begin{tikzcd}
            {F(x, 1_{y})} & {F(1_{y})} \\
            {F(u, 1_{y})} & {F(u)}
            \arrow["{\pi_{y}}", from=1-1, to=1-2]
            \arrow["{\varphi_{u} \times 1}"', from=1-1, to=2-1]
            \arrow["{\mu_{(u, 1_{y})}}"', from=2-1, to=2-2]
            \arrow["{\chi_{u}}"', from=2-2, to=1-2]
        \end{tikzcd}
        \qquad 
        \begin{tikzcd}
            {F(u)} \\
            {F(u, 1_{y})} & {F(u)}
            \arrow["{\langle \varphi_{u}s_{u}, \chi_{u} \rangle}"', from=1-1, to=2-1]
            \arrow["{1_{F(u)}}", from=1-1, to=2-2]
            \arrow["{\mu_{(u, 1_{y})}}"', from=2-1, to=2-2]
        \end{tikzcd}
    \end{equation*}
    Altogether, these diagrams are equivalent to stating that for each $\alpha \in F(u)$ there exists a unique $\chi_{u}(\alpha) \in F(1_{y})$ such that $\mu_{(u, 1_{y})}(\varphi_{u}s_{u}(\alpha), \chi_{u}(\alpha)) = \alpha$. 
    Under the equivalence $\el \colon \Idx(\Cat, \SMULT) \simeq \Lens$, this means that the chosen lifts of the delta lens $\el(\B, F) \to \B$ are \emph{weakly opcartesian}, which by Example~\ref{example:split-opfibration} characterises the delta lens as a split opfibration. 
\end{proof}

One novelty of this characterisation of split opfibrations is that it is much closer to the characterisation of functors as lax double functors $\LO(\B) \to \SPAN$ rather than as normal lax double functors $\LO(\B) \to \PROF$ into the double category $\PROF$ of categories, functors, and profunctors. 
Proposition~\ref{proposition:split-opfibration-as-ismf} may also be understood as providing an \emph{indexed version} of the characterisation of internal split opfibrations as internal delta lenses \cite[Proposition~5.4]{Clarke2020b}.

\subsection{Delta lenses as normal lax double functors}
For a double category $\DD$ with local coequalisers preserved by loose composition, there is a double category $\MOD(\DD)$ whose objects are loose monads, whose tight morphisms are monad morphisms, and whose loose morphisms are \emph{bimodules} of monads \cite{Shulman2008}. 
Cruttwell and Shulman \cite{CruttwellShulman2010} showed that there is a correspondence as follows, where \emph{normal} means that the unit comparison of the lax double functor is the identity. 
\begin{center}
    lax double functors $\CC \to \DD$ 
    \quad $\leftrightsquigarrow$ \quad
    normal lax double functors $\CC \to \MOD(\DD)$
\end{center}

Let $\Idx(\Cat, \XX)_{\mathrm{n}}$ denote the full subcategory of $\Idx(\Cat, \XX)$ whose objects are normal lax double functors into $\XX$. 
We may restate the correspondence of Cruttwell and Shulman as follows. 

\begin{lemma}
\label{lemma:lax-to-normal-lax}
There is an equivalence of categories $\Idx(\Cat, \XX) \simeq \Idx(\Cat, \MOD(\XX))_{\mathrm{n}}$.
\end{lemma}

The double category $\MOD(\SPAN)$ is equivalent to the double category $\PROF$, since loose monads are small categories, tight monad morphisms are functors, and bimodules of monads are profunctors.
Applying Lemma~\ref{lemma:lax-to-normal-lax} and Proposition~\ref{proposition:category-of-elements-SPAN}, we recover the well-known characterisation of functors into $\B$ as normal lax (double) functors $\LO(\B) \to \PROF$ due to Bénabou. 

\begin{corollary}
    There is an equivalence of categories $\Cat^{\two} \simeq \Idx(\Cat, \PROF)_{\mathrm{n}}$. 
\end{corollary}

The double category $\SMULT$ has local coequalisers preserved by loose composition, since both $\SQ(\Set)$ and $\SPAN$ admit local coequalisers and $K_{\ast} \colon \SQ(\Set) \to \SPAN$ preserves them. 
Therefore, we can also characterise delta lenses as certain normal lax double functors. 

\begin{proposition}
\label{proposition:delta-lens-as-lndf}
    There is an equivalence $\Lens \simeq \Idx(\Cat, \MOD(\SMULT))_{\mathrm{n}}$.
\end{proposition}
\begin{proof}
    Follows immediately from Theorem~\ref{theorem:main} and Lemma~\ref{lemma:lax-to-normal-lax}. 
\end{proof}

What are the objects, tight morphisms, and loose morphisms in $\MOD(\SMULT)$? 
The double functor $U_{2} \colon \SMULT \to \SPAN$, defined in \eqref{equation:SMULT-to-SPAN}, is faithful, therefore we may view $\SMULT$ as like $\SPAN$ with additional structure on the loose morphisms. 
We may also show that $\MOD(\SMULT)$ can be viewed as $\PROF$ with additional structure on the loose morphisms. 

\begin{proposition}
    The double functor $\MOD(U_{2}) \colon \MOD(\SMULT) \to \MOD(\SPAN)$ is the identity on objects and tight morphisms, and faithful. 
\end{proposition}

Therefore, a loose monad in $\SMULT$ is a small category, and a monad morphism is a functor. 
A bimodule of monads is a triple $(p, f, \varphi)$ which consists of a profunctor $p \colon \A \to \B$, corresponding to a functor $\A^{\op} \times \B \to \Set$, a functor $f \colon \A_{0} \to \B_{0}$ between discrete categories, corresponding to a function $f \colon \obj(\A) \to \obj(\B)$, and a cell in $\PROF$
\begin{equation*}
    \begin{tikzcd}
        {\A_{0}} & {\B_{0}} \\
        \A & \B
        \arrow[""{name=0, anchor=center, inner sep=0}, "{f_{\ast}}", "\shortmid"{marking}, from=1-1, to=1-2]
        \arrow["{\iota_{\A}}"', from=1-1, to=2-1]
        \arrow["{\iota_{\B}}", from=1-2, to=2-2]
        \arrow[""{name=1, anchor=center, inner sep=0}, "p"'{inner sep = 5pt}, "\shortmid"{marking}, from=2-1, to=2-2]
        \arrow["\varphi"{description}, draw=none, from=0, to=1]
    \end{tikzcd}
\end{equation*}
where $f_{\ast} \colon \A_{0} \lto \B_{0}$ is the representable profunctor such that $f_{\ast}(a, b)$ is the singleton if $fa = b$ and the empty set otherwise, and $\iota_{\A}$ and $\iota_{\B}$ are the canonical identity-on-objects functors.
The cell $\varphi$ above amounts to the choice of an element $\varphi_{a} \in p(a, fa)$ for each object $a \in \A$. 
A cell in $\MOD(\SMULT)$ is a cell in $\PROF$ which commutes with this additional structure on loose morphisms. 

There is a double functor $K_{\ast}' \colon \SQ(\Cat) \to \MOD(\SMULT)$ which is the identity on objects and tight morphisms, and whose assignment on loose morphisms sends a functor $f \colon \A \to \B$ to the triple $(f_{\ast}, f_{0}, \eta)$ where $f_{\ast} \colon \A \to \B$ is given by $f_{\ast}(a, b) = \B(fa, b)$, $f_{0} \colon \A_{0} \to \B_{0}$ is determined by the underlying object assignment of $f$, and $\eta_{a} \in \B(fa, fa)$ chooses the element $1_{fa}$. 

\begin{corollary}
\label{corollary:split-opfibration-nldf}
    A normal lax double functor $F \colon \LO(\B) \to \MOD(\SMULT)$ corresponds to a split opfibration under the equivalence $\Lens \simeq \Idx(\Cat, \MOD(\SMULT))_{\mathrm{n}}$ if and only if it factors through $K_{\ast}' \colon \SQ(\Cat) \to \MOD(\SMULT)$. 
\end{corollary}
\begin{proof}
    Follows immediately from Proposition~\ref{proposition:classical-Grothendieck}. 
\end{proof}

Despite the characterisation of delta lenses as normal lax double functors in Proposition~\ref{proposition:delta-lens-as-lndf} being arguably more complex than the characterisation as lax double functors in Theorem~\ref{theorem:main}, the corresponding characterisation of split opfibrations in Corollary~\ref{corollary:split-opfibration-nldf} is much simpler than that of Proposition~\ref{proposition:split-opfibration-as-ismf}. 

\subsection{Pullback and pushforward of indexed split multivalued functions}

There is a forgetful functor $\pi \colon \Idx(\Cat, \SMULT) \to \Cat$ which sends $(\B, F)$ to $\B$; under the equivalence of Theorem~\ref{theorem:main} this corresponds to $\cod \colon \Lens \to \Cat$ which sends a delta lens $(f, \varphi) \colon A \to \B$ to its codomain. 

\begin{proposition}
    The functor $\pi \colon \Idx(\Cat, \SMULT) \to \Cat$ has both a left adjoint and right adjoint. 
\end{proposition}
\begin{proof}
    The left adjoint sends each category $\B$ to the constant strict double functor $\emptyset \colon \LO(\B) \to \SMULT$ which chooses the empty set.
    The right adjoint sends each category $\B$ to the constant strict double functor $\ast \colon \LO(\B) \to \SMULT$ which chooses the terminal set. 
\end{proof}

The pullback of a delta lens along a functor (which may also have a delta lens structure) has been important in the study of \emph{symmetric delta lenses} \cite{DiMeglio2023,JohnsonRosebrugh2015}. 
In the context of indexed split multivalued functions, the notion of pullback is very easy to capture. 

\begin{proposition}
\label{proposition:pi-fibration}
    The functor $\pi \colon \Idx(\Cat, \SMULT) \to \Cat$ is a fibration.
\end{proposition}
\begin{proof}
    Given an object $(\B, F \colon \LO(\B) \to \SMULT)$ in $\Idx(\Cat, \SMULT)$ and a morphism $k \colon \D \to \B$ in $\Cat$, there is a cartesian lift  
    \begin{equation*}
        \begin{tikzcd}[row sep=tiny]
            {\LO(\D)} \\
            & \SMULT \\
            {\LO(\B)}
            \arrow[""{name=0, anchor=center, inner sep=0}, "{\LO(k) \circ F}", from=1-1, to=2-2]
            \arrow["{\LO(k)}"', from=1-1, to=3-1]
            \arrow[""{name=1, anchor=center, inner sep=0}, "F"', from=3-1, to=2-2]
            \arrow[""{name=1p, anchor=center, inner sep=0}, phantom, from=3-1, to=2-2, start anchor=center, end anchor=center]
            \arrow["{\id }"', shorten <=4pt, shorten >=4pt, Rightarrow, from=0, to=1p]
        \end{tikzcd}
    \end{equation*}
    given by pre-composition of $F$ by $\LO(k)$.
\end{proof}

Surprisingly, there is also a notion of pushforward of a delta lens along a functor out of its codomain. 
In the setting of double categories, this amounts to the left Kan extension of a lax double functor $F \colon \LO(\B) \to \SMULT$ along a strict double functor $\LO(k) \colon \LO(\B) \to \LO(\D)$. 
Although Left Kan extensions in this setting have been considered by Grandis and Paré \cite{GrandisPare2007}, it is unclear if $\SMULT$ has enough colimits in general. 
Fortunately, we are able to work in the setting of diagrammatic delta lenses where the pushforward is relatively easy to compute. 

\begin{proposition}
\label{proposition:pi-opfibration}
    The functor $\pi \colon \Idx(\Cat, \SMULT) \to \Cat$ is an opfibration. 
\end{proposition}
\begin{proof}
    The equivalence $\DiaLens \simeq \Idx(\Cat, \SMULT)$ is a fibrewise equivalence with respect to the functor $\pi \colon \Idx(\Cat, \SMULT) \to \Cat$, so it is enough to show that the functor $\cod \colon \DiaLens \to \Cat$ is an opfibration. 

    Given an object $(f \colon \A \to \B, p \colon \X \to \A)$ in $\DiaLens$ and a morphism $g \colon \B \to \C$ in $\Cat$, the opcartesian lift may be computed as follows. 
    \begin{equation*}
        \begin{tikzcd}
            \X & \Y \\
            \A & {\A +_{\X} \Y} \\
            \B & \C
            \arrow["{\text{initial}}", from=1-1, to=1-2]
            \arrow["p", from=1-1, to=2-1]
            \arrow["fp"', curve={height=30pt}, from=1-1, to=3-1]
            \arrow["{\text{i.o.o}}"', from=1-2, to=2-2]
            \arrow["{\text{d. opf}}", curve={height=-30pt}, from=1-2, to=3-2]
            \arrow[from=2-1, to=2-2]
            \arrow["f", from=2-1, to=3-1]
            \arrow["\lrcorner"{anchor=center, pos=0.125, rotate=180}, draw=none, from=2-2, to=1-1]
            \arrow[dashed, from=2-2, to=3-2]
            \arrow["g"', from=3-1, to=3-2]
        \end{tikzcd}
    \end{equation*}
    First, we may consider the composite functor $gfp \colon \X \to \C$ and take its comprehensive factorisation, giving an initial functor $\X \to \Y$ and a discrete opfibration $h \colon \Y \to \C$. 
    Next, we can take the pushout of $p \colon \X \to \A$ along the initial functor to obtain an identity-on-objects functor $q \colon \Y \to \A +_{\X} \Y$ since $p$ is identity-on-objects and these are stable under pushout. 
    Finally, using the universal property of the pushout, we obtain a functor $[gf, h] \colon \A +_{\X} \Y \to \C$, which when precomposed with the identity-on-objects functor $q \colon \Y \to \A +_{\X} \Y$ yields a discrete opfibration $h$. 
    
    Therefore, we have a candidate opcartesian morphism $(f, p) \to ([gf, h], q)$ in $\DiaLens$. 
    To show that has the appropriate universal property, one uses the orthogonality property of the comprehensive factorisation system, and the universal property of the pushout; we omit the straightforward diagram-chasing. 
\end{proof}

By Proposition~\ref{proposition:pi-fibration} and Proposition~\ref{proposition:pi-opfibration}, the functor $\pi \colon \Idx(\Cat, \SMULT) \to \Cat$ is a \emph{bifibration}; we let $[\LO(\B), \SMULT]_{\lax}$ denote its fibre over an object $\B$ in $\Cat$. 

\begin{corollary}
    For each functor $f \colon \A \to \B$, there is an adjunction
    \[
        \begin{tikzcd}
            {[\LO(\B), \SMULT]_{\lax}} & {[\LO(\A), \SMULT]_{\lax}}
            \arrow[""{name=0, anchor=center, inner sep=0}, "{\Delta_{f}}"', shift right=2, from=1-1, to=1-2]
            \arrow[""{name=1, anchor=center, inner sep=0}, "{\Sigma_{f}}"', shift right=2, from=1-2, to=1-1]
            \arrow["\dashv"{anchor=center, rotate=-90}, draw=none, from=1, to=0]
        \end{tikzcd}
    \]
\end{corollary}
\begin{proof}
    The right adjoint $\Delta_{f}$ is given by taking the cartesian lift of the morphism $f \colon \A \to \B$, while $\Sigma_{f}$ is given by taking the opcartesian lift of $f$. 
    Since $\pi \colon \Idx(\Cat, \SMULT) \to \Cat$ is bifibration, it follows immediately that these functors are adjoint. 
\end{proof}

\subsection{Characterising retrofunctors as indexed split multivalued functions}

A retrofunctor is a kind of morphism of categories first introduced under the name \emph{cofunctor} \cite{Aguiar1997}, and shares a close relationship with delta lenses. 

\begin{definition}
\label{definition:retrofunctor}
    Given categories $\A$ and $\B$, a \emph{retrofunctor} $(f, \varphi) \colon \A \lto \B$ consists of a function $f \colon \obj(\A) \to \obj(\B)$ equipped with a lifting operation
    \[
        (a \in \A, u \colon fa \to b \in \B) 
        \quad \longmapsto \quad
        \varphi(a, u) \colon a \to a' \in \A 
    \]
    such that the following axioms hold, where $\cod(-)$ denotes the codomain.
    \begin{enumerate}[label=(R\arabic*)]
    \itemsep=1ex
        \item \quad $f \cod(\varphi(a, u)) = \cod(u)$; 
        \item \quad $\varphi(a, 1_{fa}) = 1_{a}$; 
        \item \quad $\varphi(a, v \circ u) = \varphi(a', v) \circ \varphi(a, u)$. 
    \end{enumerate}
\end{definition}

This is almost identical to Definition~\ref{definition:delta-lens} for a delta lens, with the key difference being that the functor $\A \to \B$ has been replaced with a function $\obj(\A) \to \obj(\B)$ and the first axiom has been weakened accordingly. 

Although retrofunctors are typically understood as morphisms in a category (or a double category), here we will treat them as objects of a category. 

\begin{definition}
\label{definition:category-Ret}
    Let $\RetFun$ denote the category whose objects are retrofunctors, and whose morphisms $(h, k) \colon (f, \varphi) \to (g, \psi)$, as depicted below, consist of a pair of functors such that $kf(a) = gh(a)$ and $h\varphi(a,u) = \psi(ha, ku)$. 
    \begin{equation*}
        \begin{tikzcd}
            \A & \C \\
            \B & \D
            \arrow["h", from=1-1, to=1-2]
            \arrow["{{(f, \varphi)}}"', "\shortmid"{marking}, from=1-1, to=2-1]
            \arrow["{{(g, \psi)}}", "\shortmid"{marking}, from=1-2, to=2-2]
            \arrow["k"', from=2-1, to=2-2]
        \end{tikzcd}
    \end{equation*}
\end{definition}

There is a faithful functor $U \colon \Lens \to \RetFun$ which sends a delta lens $(f, \varphi)$ to its underlying retrofunctor $(f_{0}, \varphi)$ where $f_{0}$ is the underlying object assignment of the functor $f$. 

\begin{proposition}[{\cite[Theorem~9]{Clarke2021}}]
\label{proposition:comonadicity}
    The functor $U \colon \Lens \to \RetFun$ is comonadic. 
\end{proposition}

The right adjoint $R \colon \RetFun \to \Lens$ sends a retrofunctor $(f, \varphi) \colon \A \to \B$ to the delta lens $(Rf, R\varphi) \colon \X \to \B$ where $\X$ has the same objects as $\A$ and whose morphisms are pairs $(w \colon a \to a' \in \A, u \colon fa \to fa' \in \B) \colon a \to a'$. 
The functor $Rf \colon \X \to \B$ is given by $f$ on objects, and projection in the second component on morphisms. 
Given an object $a \in \X$ and a morphism $u \colon fa \to b$ in $\B$, the chosen lift $R\varphi(a, u)$ is the morphism $(\varphi(a, u), u) \colon a \to a'$ in $\X$. 

\begin{definition}
    A \emph{cofree delta lens} is a delta lens in the image of the right adjoint $R \colon \RetFun \to \Lens$. 
\end{definition}

The relationship between delta lenses and retrofunctors is far richer than the adjunction between their respective categories. 
For example, there is a correspondence between retrofunctors and certain spans of functors \cite{Clarke2020}, leading to an equivalence of categories analogous to $\Lens \simeq \DiaLens$ in Theorem~\ref{theorem:Lens-DiaLens}. 
Similarly, there is a correspondence between retrofunctors and certain globular transformations ``over'' the double functor $K_{\ast} \colon \SQ(\Set) \to \SPAN$, leading to an equivalence of categories analogous to $\Lens \simeq \GlobCone(\Cat, K_{\ast})$ for delta lenses. 
Instead of proving this result in detail, which would take us outside the scope of this paper, we show how each retrofunctor determines a globular transformation of double functors, and provide a characterisation of the cofree delta lenses. 

Given a category $\B$, let $\B_{\infty}$ denote the \emph{codiscrete category} determined by its underlying set of objects, and $j_{\B} \colon \B \to \B_{\infty}$ the canonical identity-on-objects functor. 

\begin{proposition}
    Each retrofunctor $(f, \varphi) \colon \A \lto \B$ determines a globular transformation of lax double functors as follows.
    \begin{equation}
    \label{equation:retrofunctor-to-GlobCone}
        \begin{tikzcd}
            {\LO(\B)} & {\LO(\B_{\infty})} \\
            {\SQ(\Set)} & \SPAN
            \arrow["{\LO(j_{\B})}", from=1-1, to=1-2]
            \arrow[""{name=0, anchor=center, inner sep=0}, "{F_{1}}"', from=1-1, to=2-1]
            \arrow[""{name=1, anchor=center, inner sep=0}, "{F_{2}}", from=1-2, to=2-2]
            \arrow["{K_{\ast}}"', from=2-1, to=2-2]
            \arrow["\Phi", shorten <=30pt, shorten >=30pt, Rightarrow, from=0, to=1]
        \end{tikzcd}
    \end{equation}
\end{proposition}
\begin{proof}
    Since $\Phi$ is a globular transformation, $F_{1}$ and $F_{2}$ have the same assignment on objects which we denote by $F$. 
    Let $F(x) = \{ a \in \A \mid fa = b \}$ for each object $x \in \B$. 
    Given a morphism $u \colon x \to y$, we define the function $F_{1}(u) \colon F(x) \to F(y)$ by the assignment $a \mapsto \cod(\varphi(a, u))$; the (strict) double functor $F_{1} \colon \LO(\B) \to \SQ(\Set)$ is well-defined by the axioms of retrofunctor.
    Given a pair of objects $(x, y)$, we define the set $F_{2}(x, y) = \{ w \colon a \to a' \in \A \mid fa = x, fa' = y \}$ and a span 
    \[
        \begin{tikzcd}
            {F(x)} & {F(x, y)} & {F(y)}
            \arrow["{s_{x}}"', from=1-2, to=1-1]
            \arrow["{t_{y}}", from=1-2, to=1-3]
        \end{tikzcd}
    \]
    where $s_{x}(w \colon a \to a') = a$ and $t_{y}(w \colon a \to a') = a'$. 
    This extends to a lax double functor $F_{2} \colon \LO(\B_{\infty}) \to \SPAN$, where the unit and multiplication comparison cells determined by the identity and composition of $\A$. 
    Finally, the globular transformation $\Phi$ has the component at a morphism $u \colon x \to y$ in $\B$ given by the diagram
    \[
        \begin{tikzcd}
            {F(x)} & {F(x)} & {F(y)} \\
            {F(x)} & {F(x, y)} & {F(y)}
            \arrow[equals, from=1-1, to=2-1]
            \arrow["{1_{F(x)}}"', from=1-2, to=1-1]
            \arrow["{F_{1}(u)}", from=1-2, to=1-3]
            \arrow["{\Phi_{u}}", from=1-2, to=2-2]
            \arrow[equals, from=1-3, to=2-3]
            \arrow["{s_{x}}"', from=2-2, to=2-1]
            \arrow["{t_{y}}", from=2-2, to=2-3]
        \end{tikzcd}
    \]
    where $\Phi_{u}(a) = \varphi(a, u)$. 
\end{proof}

Given a globular transformation \eqref{equation:retrofunctor-to-GlobCone}, we may easily obtain the globular cone over $K_{\ast}$ given the triple $(\B, F_{1}, F_{2} \circ \LO(j_{B}))$.
This describes the assignment on objects of the functor $\RetFun \to \Lens \simeq \GlobCone(\Cat, K_{\ast})$ and provides a way of characterising the objects in the essential image.

\begin{proposition}
    A globular cone
    \begin{equation*}
    \begin{tikzcd}[column sep = small]
        & \LO(\B) \\
        {\SQ(\Set)} && \SPAN
        \arrow[""{name=0, anchor=center, inner sep=0}, "{F_{1}}"', from=1-2, to=2-1]
        \arrow[""{name=1, anchor=center, inner sep=0}, "{F_{2}}", from=1-2, to=2-3]
        \arrow["K_{\ast}"', from=2-1, to=2-3]
        \arrow["\varphi", shift right=2, shorten <=10pt, shorten >=10pt, Rightarrow, from=0, to=1]
    \end{tikzcd}
    \end{equation*}
    corresponds to a cofree delta lens under the equivalence $\Lens \simeq \GlobCone(\Cat, K_{\ast})$ if and only if $F_{2}$ factors through the double functor $\LO(j_{\B}) \colon \LO(\B) \to \LO(\B_{\infty})$. 
\end{proposition}

\subsection{The category of elements as a right adjoint}
In Proposition~\ref{proposition:comonadicity}, we recalled that delta lenses are coalgebras for a comonad on the category of retrofunctors. 
There is also a dual result, which states that delta lenses are algebras for a monad on the category of functors. 

\begin{proposition}[{\cite[Corollary~24]{Clarke2023}}]
\label{proposition:monadicity}
    The functor $U \colon \Lens \to \Cat^{\two}$ is monadic.
\end{proposition}

The category of elements of an indexed split multivalued function describes a functor $\el \colon \Idx(\Cat, \SMULT) \to \Cat$ which given by following the composite. 
\begin{equation*}
    \begin{tikzcd}
        {\Idx(\Cat, \SMULT)} & \Lens & {\Cat^{\two}} & \Cat
        \arrow["\simeq", from=1-1, to=1-2]
        \arrow["U", from=1-2, to=1-3]
        \arrow["\dom", from=1-3, to=1-4]
    \end{tikzcd}
\end{equation*}
The functor $\dom \colon \Cat^{\two} \to \Cat$ which sends a functor to its domain category has a left adjoint, which sends a category to its identity functor.
The functor $U \colon \Lens \to \Cat^{\two}$ has a left adjoint by Proposition~\ref{proposition:monadicity}. 
Finally, $\el \colon \Idx(\Cat, \SMULT) \to \Lens$ is an (adjoint) equivalence, whose left adjoint is described explicitly in Section~\ref{sec:proof-sketch}. 
Thus, we can characterise the category of elements of an indexed split multivalued function as right adjoint to the functor which constructs the indexed multivalued valued function corresponding to the free delta lens on the identity functor. 

\begin{proposition}
    The functor $\el \colon \Idx(\Cat, \SMULT) \to \Cat$ has a left adjoint. 
\end{proposition}
\begin{proof}
    We may define the left adjoint $\mathrm{Fr} \colon \Cat \to \Idx(\Cat, \SMULT)$ explicitly. 
    Given a category $\B$, the lax double functor $\mathrm{Fr}_{\B} \colon \LO(\B) \to \SMULT$ is given as follows.
    For each object $x \in \B$, we define the set $\mathrm{Fr}_{\B}(x) = \sum_{b \in \B} \B(b, x)$ and for each morphism $u \colon x \to y$ in $\B$ we define the set 
    \[
        \mathrm{Fr}_{\B}(u) = \mathrm{Fr}_{\B}(x) + \{ (b_{1} \xrightarrow{\alpha} x, x \xrightarrow{\beta} b_{2}, b_{2} \xrightarrow{\gamma} b_{3}, b_{3} \xrightarrow{\delta} y) \mid \beta \alpha = 1_{b_{1}}, \delta \gamma \beta = u, \gamma \neq 1 \}.
    \]
    To construct the split multivalued function 
    \[
        \begin{tikzcd}
            {\mathrm{Fr}_{\B}(x)} & {\mathrm{Fr}_{\B}(u)} & {\mathrm{Fr}_{\B}(y)}
            \arrow["{\sigma_{u}}", shift left=2, tail, from=1-1, to=1-2]
            \arrow["{s_{u}}", shift left=2, two heads, from=1-2, to=1-1]
            \arrow["{t_{u}}", from=1-2, to=1-3]
        \end{tikzcd}
    \]
    we define the following functions. 
    \begin{align*}
        s_{u}(\alpha \colon b \to x) &= \alpha \quad \text{ and } \quad s_{u}(\alpha, \beta, \gamma, \delta) = \alpha \\
        t_{u}(\alpha \colon b \to x) &= u \circ \alpha \quad \text { and } \quad s_{u}(\alpha, \beta, \gamma, \delta) = \delta \\
        \sigma_{u}(\alpha \colon b \to x) &= \alpha
    \end{align*}
    Defining the function $\mu_{(u, v)} \colon F(u, v) \to F(v \circ u)$ corresponding to the multiplication comparison of the lax double functor is fairly complex, and we refer the reader to \cite[Construction~3.16]{Clarke2024} where they may infer the details.
\end{proof}

\subsection{Monoidal structures on indexed split multivalued functions}

The double category $\SMULT$ admits a cartesian and a cocartesian monoidal structure, since $\SQ(\Set)$ and $\SPAN$ have products and coproducts and the double functor $K_{\ast}$ preserves them.  
These induce corresponding monoidal structures on category $\Idx(\Cat, \SMULT)$ and the fibres $[\LO(\B), \SMULT]_{\lax}$ of the forgetful functor $\pi \colon \Idx(\Cat, \SMULT) \to \Cat$. 

\begin{proposition}
    The category $\Idx(\Cat, \SMULT)$ has products and coproducts.  
\end{proposition}
\begin{proof}
    Given lax double functors $F \colon \LO(\B) \to \SMULT$ and $G \colon \LO(\D) \to \SMULT$, their product is given by the lax double functor $F \times G \colon \LO(\B \times \D) \to \SMULT$ where 
    $(F \times G)(x, x') = F(x) \times G(x')$ for each object $(x, x') \in \B \times \D$, 
    and $(F \times G)(u, u')$ is the split multivalued function
    \begin{equation*}
        \begin{tikzcd}
            {F(x)\times G(x') } & {F(u)\times G(u')} & {F(y)\ \times G(y')}
            \arrow["{\sigma_{u} \times \sigma_{u'}}", shift left=2, tail, from=1-1, to=1-2]
            \arrow["{s_{u} \times s_{u'}}", shift left=2, two heads, from=1-2, to=1-1]
            \arrow["{t_{u} \times t_{u'}}", from=1-2, to=1-3]
        \end{tikzcd}
    \end{equation*}
    for each morphism $(u \colon x \to y, u' \colon x' \to y')$ in $\B \times \D$. 
    Dually, the coproduct $F + G$ is given by the component-wise coproduct on $\LO(\B + \D)$.  
\end{proof}

\begin{proposition}
    The category $[\LO(\B), \SMULT]_{\lax}$ has products and coproducts. 
\end{proposition}
\begin{proof}
    Given lax double functors $F \colon \LO(\B) \to \SMULT$ and $G \colon \LO(\B) \to \SMULT$, their product is given by the lax double functor $F \times_{\B} G \colon \LO(\B) \to \SMULT$ where 
    $(F \times_{\B} G)(x) = F(x) \times G(x)$ for each object $x \in \B$, 
    and $(F \times_{\B} G)(u)$ is the split multivalued function
    \[
    \begin{tikzcd}
        {F(x)\times G(x) } & {F(u)\times G(u)} & {F(y)\ \times G(y)}
        \arrow["{\sigma_{u} \times \sigma_{u}}", shift left=2, tail, from=1-1, to=1-2]
        \arrow["{s_{u} \times s_{u}}", shift left=2, two heads, from=1-2, to=1-1]
        \arrow["{t_{u} \times t_{u}}", from=1-2, to=1-3]
    \end{tikzcd}
    \]
    for each morphism $u \colon x \to y$ in $\B$. 
    Dually, the coproduct $F +_{\B} G$ is given by the component-wise coproduct. 
\end{proof}

The product in $[\LO(\B), \SMULT]_{\lax}$ corresponds to the fibre product of delta lenses which has been used in the study of symmetric delta lenses \cite[Proposition~3.4]{Clarke2021b}.

\subsection{Sub-double categories and classes of delta lenses}

We can capture many natural classes of delta lenses under the equivalence $\Lens \simeq \Idx(\Cat, \SMULT)$ by restricting the lax double functors $\LO(\B) \to \SMULT$ which factor through some (fully faithful) double functor $\XX \to \SMULT$.

We have already seen examples of this kind in Section~\ref{sec:Adjunctions}. 
An indexed split multivalued function $\LO(\B) \to \SMULT$ corresponds to: 
\begin{itemize}
    \item a discrete opfibration if it factors through $\langle 1, K_{\ast} \rangle \colon \SQ(\Set) \to \SMULT$ defined in Proposition~\ref{proposition:SQSET-SMULT-adjunction};
    \item a fully faithful delta lens if it factors through $\SQ(\Set) \to \SMULT$ defined in Proposition~\ref{proposition:reflective-adjunction-SQSET-SMULT};
    \item an injective-on-objects delta lens if it factors through $\SQ(2) \to \SMULT$ defined in Corollary~\ref{corollary:SQ2-SQSET-SMULT}.
\end{itemize}

Let $\mathrm{\mathbb{P}Set} \to \SMULT$ denote the full sub-double category with a single object given by the singleton set. 
Loose morphisms in $\mathrm{\mathbb{P}Set}$ are \emph{pointed sets} and their composition given by the product of pointed sets. 

Let $\SMULT_{\neq \emptyset} \to \SMULT$ denote the full sub-double category whose objects are \emph{non-empty sets}. 

Let $\mathrm{\mathbb{R}elSMult} \to \SMULT$ denote the full sub-double category with the same objects and whose loose morphisms are given by split multivalued functions whose underlying span is a \emph{relation}. 

Let $\mathrm{\mathbb{S}Mono} \to \SMULT$ denote the full sub-double category with the same objects and whose loose morphisms are given by \emph{split monomorphisms}; that is, split multivalued functions of the following form. 
\begin{equation*}
    \begin{tikzcd}
        A & B & B
        \arrow["\sigma", shift left=2, tail, from=1-1, to=1-2]
        \arrow["s", shift left=2, two heads, from=1-2, to=1-1]
        \arrow["{1_{B}}", from=1-2, to=1-3]
    \end{tikzcd}
    \qquad s \circ \sigma = 1_{A}
\end{equation*}

\begin{proposition}
    A lax double functor $\LO(\B) \to \SMULT$ corresponds to: 
    \begin{itemize}
        \item a bijective-on-objects delta lens if it factors through $\mathrm{\mathbb{P}Set} \to \SMULT$;
        \item a surjective-on-objects delta lens if it factors through $\SMULT_{\neq \emptyset} \to \SMULT$;
        \item a faithful delta lens if it factors though $\mathrm{\mathbb{R}elSMult} \to \SMULT$;
        \item a delta lens whose underlying functor is a \emph{discrete fibration} if it factors through $\mathrm{\mathbb{S}Mono} \to \SMULT$.
    \end{itemize}    
\end{proposition}

%% file: references.tex
\section{Acknowledgements}
This work is based on Chapter~4 of the author's PhD thesis \cite{Clarke2022}. 
The author first spoke about this work at the Australian Category Seminar in September 2020, and later at the Calgary Peripatetic Seminar in March 2021 and the Applied Category Theory conference in June 2024. 
Thank you to everyone who gave feedback on this research during its development.